\documentclass[a4paper]{article}
\usepackage[top=25truemm,bottom=25truemm,left=20truemm,right=20truemm]{geometry}
\usepackage{amsmath,amsthm,amsfonts,amssymb,stmaryrd,indentfirst,tabularray,makecell,hhline,comment,mathtools}
\usepackage[mathscr]{euscript}
\usepackage[shortlabels]{enumitem}
\usepackage[labelfont=rm]{caption}
\usepackage{titlesec}
\usepackage{tikz-cd}
\usepackage{url}
\usepackage[backref]{hyperref}
\hypersetup{
    colorlinks = true,
    linkcolor = [rgb]{0.3,0.55,0.25},
    citecolor = [rgb]{0.164,0.39,0.77},
    urlcolor = [rgb]{0.7,0.2,0.4}
}

\DeclareMathOperator{\id}{id}
\DeclareMathOperator{\Ker}{Ker}
\DeclareMathOperator{\Aut}{Aut}
\DeclareMathOperator{\ad}{ad}

\theoremstyle{plain}
\newtheorem{theorem}{Theorem}
\newtheorem*{theorem*}{Theorem}

\newtheorem{lemma}[theorem]{Lemma}

\theoremstyle{definition}
\newtheorem{definition}[theorem]{Definition}
\newtheorem{remark}[theorem]{Remark}
\newtheorem{example}[theorem]{Example}

\numberwithin{theorem}{section}
\everymath{\displaystyle}
\allowdisplaybreaks

\titleformat{\section}{\large\scshape}{\thesection.}{3pt}{}
\titleformat{\subsection}{\scshape}{\thesubsection}{3pt}{}

\usepackage{apptools}
\titleformat{\section}{\large\scshape}{\IfAppendix{\appendixname}{} \thesection.}{3pt}{}
\titleformat{\subsection}{\scshape}{\thesubsection.}{3pt}{}

\usepackage{titling}  
\pretitle{
	\begin{center}\Large\bfseries
}
\posttitle{
	\end{center}\ \\[-20pt]
}
\preauthor{
	\begin{center}\large
}
\postauthor{
	\end{center}\ \\[-60pt]
}

\title{The Framed Version of the Universal KZB Connection\\in Higher Genera}
\author{Toyo TANIGUCHI \thanks{Graduate School of Mathematical Sciences, The University of Tokyo. 3-8-1, Komaba, Meguro-ku, Tokyo, 153-8914, Japan. E-mail: \texttt{toyo(at)ms.u-tokyo.ac.jp}}}
\date{}

\begin{document}
\maketitle

\centerline{\textit{With Appendix by Benjamin Enriquez and Toyo Taniguchi}}

\begin{abstract}
\noindent In this paper, the universal KZB connection on the configuration space of points on a closed Riemann surface of an arbitrary genus, introduced by Enriquez, is lifted to the configuration space of points with tangent vectors. This lifted connection is then used to prove the existence of a higher genus version of a Drinfeld associator in the sense of Gonzalez.\\
\end{abstract}

\noindent{\textit{MSC2020: 53C10 (primary); 16T25, 16W70, 18M75, 20F36, 20F40, 57K20, 57M05.}

\section{Introduction}

When Drinfeld associators were introduced in his seminal paper \cite{drinfeld}, its existence was proved using the universal Knizhnik--Zamolodchikov (KZ) connection, which is a particular flat connection on the configuration space $\mathrm{Conf}_n(\mathbb{C})$ of $n\geq 1$ points on the complex plane $\mathbb{C}$. This connection takes its value in the Drinfeld--Kohno Lie algebra $\mathfrak{t}_n$, which is a graded Lie algebra associated with the pure braid group $\mathrm{PB}_n = \pi_1(\mathrm{Conf}_n(\mathbb{C}))$. The original (non-universal) KZ connection in \cite{kz} was formulated in the context of conformal field theory, and it was generalised to higher genera by Bernard \cite{bernard1, bernard2} and Tsuchiya--Ueno--Yamada \cite{tuy}.

One notable example of a Drinfeld asssociator, called the KZ associator, is obtained as the monodromy of the KZ connection along a certain path in $\mathrm{Conf}_3(\mathbb{C})$, and its key property is the compatibility with \textit{operadic} structure of $\mathrm{Conf}_n(\mathbb{C})$: in terms of pure braids, the associated monodromy behaves naturally with the doubling of strands in $\mathrm{PB}_n$. This operadic nature led to the reformulation of Drinfeld associators by Bar-Natan \cite{barnatan}.

By replacing the complex plane $\mathbb{C}$ with an orientable closed surface $C$, say, of genus $g$, Drinfeld associators are generalised to an arbitrary genus, which was studied by Gonzalez \cite{gonzalez}. We thusly call this version of associators \textit{genus} $g$ \textit{Gonzalez--Drinfeld associators}. One important point, however, is that we have to consider the \textit{framed} version $\mathrm{Conf}^\mathrm{fr}_n(C)$ of the configuration space to recover its operadic structure. More presicely, $\mathrm{Conf}^\mathrm{fr}_n(C)$ is defined as the pull-back bundle fitting in the diagram
\[\begin{tikzcd}[cramped]
	\mathrm{Conf}^\mathrm{fr}_n(C) \arrow[r]\arrow[d] &(TC \setminus 0_{TC})^n \arrow[d]\\
	\mathrm{Conf}_n(C) \arrow[r] & C^n
\end{tikzcd}\]
where $TC$ is the tangent bundle of $C$, and $0_{TC}$ is its zero-section. This additional data of framing enables us to do various doubling operations just as the case of the complex plane $\mathbb{C}$. In the case of $g=1$, a slightly different version had already been introduced by Enriquez \cite{elliptic} under the name of \textit{elliptic associators}, and is equivalent to a genus $1$ Gonzalez--Drinfeld associator (although we could not find the proof in the literature; the proof will be given in a sequel). We note that another generalisation based on rational models of $\mathrm{Conf}^\mathrm{fr}_n(C)$, in contrast to Gonzalez' framework of groupoid models, was introduced and shown to exist by Campos--Idrissi--Willwacher \cite{ciw}. Whether they agree is still an open question to the best of author's knowledge.

The existence of Gonzalez--Drinfeld associators, on the other hand, is the main problem we discuss in this paper. For $g=1$, the existence is also shown in \cite{gonzalez}, which is based on the elliptic version of the universal KZ connection, the \textit{universal Knizhnik--Zamolodchikov--Bernard (KZB) connection} constructed by Calaque--Enriquez--Etingof \cite{cee}, while for $g\geq 2$, it is only stated as Conjecture 3.21 in \cite{gonzalez}. However, this KZB connection is further generalised for $\mathrm{Conf}_n(C)$ with the genus of $C$ arbitrary, denoted by $\alpha_\mathrm{KZ}^{(n)}$, in the work of Enriquez \cite{enriquez}, given the data of a \textit{marking} on $C$, namely a choice of a base point $*\in C$ and a set of certain generators of the fundamental group $\pi_1(C,*)$. Therefore, the last step for the construction of a genus $g$ Gonzalez--Drinfeld associator, by means of KZB-type connections, is to lift these connection to the framed configuration space such that the compatibility with the operadic structure on $\mathrm{Conf}^\mathrm{fr}_n(C)$ is satisfied. Indeed, we have the following:

\begin{theorem*}[Theorems \ref{thm:kzb}, \ref{thm:operadic}]
For each marked closed Riemann surface $C$, there exists a unique \textit{operadic} family of connections $\{\vec{\alpha}_\mathrm{KZ}^{(n)}\}_{n\geq 0}$ such that $\vec{\alpha}_\mathrm{KZ}^{(n)}$ is a flat connection on $\mathrm{Conf}^\mathrm{fr}_n(C)$ lifting Enriquez' connection $\alpha_\mathrm{KZ}^{(n)}$ along the projection $\mathrm{Conf}^\mathrm{fr}_n(C) \to \mathrm{Conf}_n(C)$.
\end{theorem*}

\noindent This gives an affirmative answer to Gonzalez' Conjecture 3.21 in \cite{gonzalez}. As a result, we obtain a \textit{genus} $g$ \textit{KZB associator}:
\begin{theorem*}[Theorems \ref{thm:kzbassoc}]
For each marked closed Riemann surface $C$, the monodromy morphism associated with $\{\vec{\alpha}_\mathrm{KZ}^{(n)}\}_{n\geq 0}$  gives rise to a genus $g$ Gonzalez--Drinfeld associator.
\end{theorem*}

The operadic nature of the framed universal KZB connection would facilitate the study of higher genus multi-zeta values; for a nice exposition with the background of various connections in higher genera, see \cite{hgmzv}. Another application was given in author's previous paper \cite{toyo}, where each genus $g$ Gonzalez--Drinfeld associator directly gives a solution to the Kashiwara--Vergne equations in genus $g$ in the sense of Alekseev--Kawazumi--Kuno--Naef \cite{akkn} in the context of the formality problem for the Goldman--Turaev Lie bialgebra.\\

\noindent\textbf{Organisation of the paper.} In Section \ref{sec:enriquez}, we review Enriquez' construction of a flat connection. Section \ref{sec:main} is devoted to the main construction of the framed KZB connection, whose operadicity is shown in Section \ref{sec:operadicity}. Section \ref{sec:gonzalez} deals with the construction of genus $g$ KZB associators. Appendix \ref{sec:correction} is a minor correction to the proof of Lemma 9 used in \cite{enriquez}.\\

\noindent\textbf{Acknowledgements.} The author is grateful to Benjamin Enriquez for heart-warming discussions on his work and improving many details, especially a great simplification of the proof of Lemma \ref{lem:phi}, and to Nariya Kawazumi for thorough reading and many helpful suggestions on the draft. This work was supported by JSPS KAKENHI Grant Number 25KJ0734.\\

\section{Enriquez' Connections} \label{sec:enriquez}

The genus $g$ version of the (non-framed) KZB connection was introduced in \cite{enriquez}. We recall some definitions and results of \cite{enriquez} in this section. First of all, we start with the genus $g$ Drinfeld--Kohno Lie algebra.

\begin{definition}
For $g\geq 0$ and a finite set $I$, the Lie algebra $\mathfrak{t}_{g,I}$ over $\mathbb{C}$ is generated by elements
\[
	\bar t_{ij}\; (i\neq j \in I)\quad\mbox{and}\quad \bar x_i^a, \bar y_i^a\; (i \in I, 1\leq a\leq g)
\]
with the relations given, for $i,j,k,l\in I$ and $1\leq a,b\leq g$, by
\begin{align*}
\begin{aligned}
	&\bar t_{ij} = \bar t_{ji},&\\
	&[\bar t_{ij}, \bar t_{kl}] = 0& &\mbox{if } \{i,j\}\cap\{k,l\}=\varnothing,\\
	&[\bar t_{ij}, \bar t_{ik} + \bar t_{jk}] = 0& &\mbox{if } \{i,j\}\cap\{k\}=\varnothing,\\
	&[\bar x_i^a, \bar y_j^b] = \delta_{ab} \bar t_{ij}& &\mbox{if } i\neq j,\\
	&[\bar x_i^a, \bar x_j^b] = [\bar y_i^a, \bar y_j^b] = 0& &\mbox{if } i\neq j,\\
	&[\bar x_k^a, \bar t_{ij}] = [\bar y_k^a, \bar t_{ij}] = 0& &\mbox{if } \{i,j\}\cap\{k\}=\varnothing,\\
	&[\bar x_i^a + \bar x_j^a, \bar t_{ij}] = [\bar y_i^a + \bar y_j^a, \bar t_{ij}] = 0
\end{aligned}
\end{align*}
and, for $i\in I$, 
\[
	\sum_{1\leq a\leq g} [\bar x_i^a, \bar y_i^a] + \sum_{j\in I\setminus\{i\}} \bar t_{ij} = 0\,.
\]
When $I = \{1,2,\dotsc, n\}$, we simply write $\mathfrak{t}_{g,n}$.
\end{definition}

\begin{definition}
Let $g\geq 0$ and $C$ a closed Riemann surface of genus $g$. For a finite set $I$, we denote the configuration space of points on $C$ labelled by $I$ by 
\[
	\mathrm{Conf}_I(C) = \{(p_i)_{i\in I} \in C^I: p_i \neq p_j \mbox{ for any }  i\neq j \in I\} =: C^I\setminus \Delta\!^{(I)}.
\]
As before, when $I = \{1,2,\dotsc, n\}$, we simply write $\mathrm{Conf}_n(C)$ and $\Delta\!^{(n)}$, respectively.
\end{definition}

\begin{definition}
Define the group $\pi_g$ by 
\[
	\pi_g = \langle A_a, B_a\, (1\leq a\leq g)| (A_1, B_1)\cdots (A_g, B_g) = 1 \rangle
\]
with the commutator $(X, Y) = XYX^{-1}Y^{-1}$. We denote by $F_g$ the free group generated by $\gamma_1,\dotsc, \gamma_g$, and put
\begin{align*}
	\sigma \colon \pi_g &\to F_g: \;A_a \mapsto 1, B_a \mapsto \gamma_a.
\end{align*}
\end{definition}

We fix $*\in C$ together with an isomorphism $\pi_1(C, *) \cong \pi_g$, to which we will refer as a \textit{marking} on $C$. Then, we obtain the covering $\varpi\colon \tilde C \to C$ associated with $\Ker(\sigma\colon \pi_g \to F_g)$ via the identification $\pi_1(C, *) \cong \pi_g$. Next, we fix $\tilde * \in \varpi^{-1}(*)$ so that $F_g$ and the group of deck transformations $\Aut(\tilde C/C)$ are identified.

The choice of $\tilde *$ gives smooth simple free loops $\mathcal{A}_a\,(1\leq a\leq g)$ on $\tilde C$ so that $\varpi(\mathcal{A}_a)$ is freely homotopic to $A_a$, and also the $F_g$-fundamental domain $\mathcal{D}$ of $\tilde C$ containing $\tilde *$ so that they satisfy
\[
	\partial \mathcal{D} = \bigcup_{1\leq a\leq g} \mathcal{A}_a \cup \gamma_a^{-1}(\mathcal{A}_a).\vspace{-5pt}
\]
An element $\gamma \in F_g$ acts on $1$-forms on $\tilde C$ by deck transformations: $\gamma (\omega) = (\gamma^{-1})^*\omega$.

\begin{definition}
We denote by $\hat{\mathfrak{t}}_{g,n}$ the pro-nilpotent completion of $\mathfrak{t}_{g,n}$. This is isomorphic to the degree-completion with respect to $\deg(\bar x_i^a) = \deg(\bar y_i^a) = 1$ and $\deg(\bar t_{ij}) = 2$.

 The flat principal $\exp(\hat{\mathfrak{t}}_{g,n})$-bundle $\mathcal{P}_n \to C^n$ is defined by
\[
	\mathcal{P}_n = \tilde C^n \times_{(F_g)^n} \exp(\hat{\mathfrak{t}}_{g,n}),
\]
where $(F_g)^n$ acts on $\tilde C^n$ by deck transformations via $(F_g)^n \cong \Aut(\tilde C/C)^n \to \Aut(\tilde C^n/C^n)$, and on $\exp(\hat{\mathfrak{t}}_{g,n})$ by
\begin{align*}
	(F_g)^n &\to \exp(\hat{\mathfrak{t}}_{g,n})\\
	\gamma_a^{(i)} &\mapsto e^{\bar x_i^a}
\end{align*}
where $\gamma_a^{(i)}$ is the $i$-th copy of $\gamma_a$ in $(F_g)^n$. We denote the associated flat connection by $\nabla_n$.
\end{definition}

Enriquez' connection $\nabla_n - \alpha^{(n)}_\mathrm{KZ}$ is constructed as a meromorphic connection on $\mathcal{P}_n$. More precisely, we will soon define an element
\[
	\alpha^{(n)}_\mathrm{KZ}\in H^0\big( C^n, \Omega^{1,0}_{C^n}\otimes (\mathcal{P}_n\times_\mathrm{ad}\hat{\mathfrak{t}}_{g,n})(\Delta\!^{(n)}) \big).
\]
which is regarded as a connection on $\tilde C$ via the natural map
\[	
	H^0\big( C^n, \Omega^{1,0}_{C^n}\otimes (\mathcal{P}_n\times_\mathrm{ad}\hat{\mathfrak{t}}_{g,n})(\Delta\!^{(n)}) \big) \to H^0\big( \tilde C^n, \Omega^{1,0}_{\tilde C^n}\otimes \hat{\mathfrak{t}}_{g,n} (\tilde \Delta\!^{(n)}) \big) 
\]
induced by the projection $\varpi^n\colon\tilde C^n \to C^n$, where we put $\tilde \Delta\!^{(n)} = (\varpi^n)^{-1}(\Delta\!^{(n)})$. To recall the definition, we shall do some preparation.

\begin{definition}\label{def:enr}
Let $\{A_a\}_{1\leq a\leq g}$ be the generators for $A$-loops of $\pi_g$ as before.
\begin{itemize}
	\item We denote by $K_C$ the canonical bundle of $C$, and put $K_C^{(i)} = \mathcal{O}_C^{\boxtimes i-1} \boxtimes K_C \boxtimes \mathcal{O}_C^{\boxtimes n-i}$, which is a sheaf on $C^n$. We define $K_{\tilde C}^{(i)}$ similarly, and denote by $d^{(i)}$ the exterior derivative applied to the $i$-th factor of $C^n$ or $\tilde C^n$.
	\item We denote by $\omega_a\in H^0(C, K_C)\, (1\leq a\leq g)$ the holomorphic $1$-forms dual to the homology classes of $\{A_a\}_{1\leq a\leq g}$.
	\item Let $\psi \in H^0\big( C^2, (K_C \boxtimes K_C) (2\Delta\!^{(2)}) \big)$ be the basic bi-differential (also called the \textit{Bergman kernel}, see p.126 of \cite{fay}, or the differential of the second kind) with the bi-residue $\frac{1}{2\pi\sqrt{-1}}$ and normalized by $\{A_a\}_{1\leq a\leq g}$, that is, for any $1\leq a\leq g$ and $q\in C\setminus A_a$,
\[
	\int_{p\in A_a} \psi(p,q) = 0.
\]
\end{itemize}
\end{definition}

If we take a local coordinate $z$ of $C$ and its copy $w$ so that $(z,w)$ is a local coordinate of $C^2$, we have
\begin{align*}
	\psi \sim (d\boxtimes d) \frac{\log(z - w)}{2\pi\sqrt{-1}} = \frac{1}{2\pi\sqrt{-1}}\frac{dz\,dw}{(z-w)^2}
\end{align*}
near the diagonal $\Delta^{(2)}$ of $C^2$, where $\sim$ denotes the equality up to holmorphic terms, and $(d\boxtimes d)$ is the induced differential on $K_C\boxtimes K_C$.

\begin{definition}
Let $s\geq 0$ and $a_1,\dotsc, a_s \in \{1,\dotsc, g\}$.
\begin{itemize}
	\item For $s\geq 1$, we put
	\[
		\delta_{a_1\dotsc a_s} = \begin{cases} 1 & a_1=\cdots =a_s \\ 0 & \textrm{otherwise}\end{cases}.
	\]
	\item The $1$-forms $\omega_{a_1\dotsc a_s} \in H^0\big( \tilde C^2, K_{\tilde C}^{(1)} (\tilde \Delta\!^{(2)}) \big)$ for $s\geq 1$ are inductively defined by
	\begin{align}\label{eq:residue}
		\gamma_a^{(1)}(\omega_{a_1\dotsc a_s}) = \sum_{0\leq k< s} \frac{1}{k!} \delta_{aa_1\dotsc a_k}\omega_{a_{k+1}\dotsc a_s},\quad \mathrm{res}_{p_1=p_2}\omega_{a_1\dotsc a_s}(p_1, p_2) = \begin{cases} -\frac{\delta_{a_1a_2}}{2\pi\sqrt{-1}} & (s=2)\\ 0 &(s\neq 2) \end{cases}
	\end{align}
	for $1\leq a\leq g$. The base case of $s=1$ is already defined in Definition \ref{def:enr}.
	\item The $1$-forms $\psi_{a_1\dotsc a_s} \in H^0\big( \tilde C^2, (K_{\tilde C}\boxtimes K_{\tilde C}) (2\tilde \Delta\!^{(2)}) \big)$ for $s\geq 0$ are defined by $\psi_{a_1\dotsc a_s} = - d^{(2)}(\omega_{a_1\dotsc a_s b b})$ for any $1\leq b\leq g$. This does not depend on the choice of $b$, and, in the case of $s=0$, it coincides with the above $\psi$.
	\item The $1$-forms $\Psi_{a_1\dotsc a_s} \in H^0\big( \tilde C^3, K_{\tilde C} ^{(1)}(\tilde\Delta\!^{(3)}) \big)$ are defined by
	\[
		\Psi_{a_1\dotsc a_s}(p, q_1, q_2) = \int_{q = q_1}^{q_2} \psi_{a_1\dotsc a_s}(p, q).
	\]
	In particular, for $s=0$, we have the expansion
	\begin{align}\label{eq:bidiff}
		\Psi \sim \frac{1}{2\pi\sqrt{-1}}\Big( \frac{dz}{z-w'} - \frac{dz}{z-w}\Big)
	\end{align}
	near the diagonal $\{z = w = w'\}$ with a local coordinate $(z,w,w')$ on $\tilde C^3$ where $w$ and $w'$ are copies of $z$.
	\item We denote projections by
	\begin{align*}
		\mathbf{p}_{i,j}&\colon \tilde C^{n+1} \to \tilde C^3\colon (p_1, \dotsc, p_n, q) \mapsto (p_i, q, p_j)\mbox{ and}\\
		\mathbf{p}_i&\colon \tilde C^{n+1} \to \tilde C^2\colon (p_1, \dotsc, p_n, q) \mapsto (p_i, q)
	\end{align*}
	for $1\leq i, j\leq n$ with $i\neq j$, and set
	\begin{align}\label{eq:omegadef}
	\begin{aligned}
		\omega_{a_1\dotsc a_s}^i &= \mathbf{p}_i^*\omega_{a_1\dotsc a_s} \in H^0\big( \tilde C^{n+1}, K_{\tilde C}^{(i)}\otimes \hat{\mathfrak{t}}_{g,n} (\tilde \Delta\!^{(n+1)}) \big),\\
		\psi_{a_1\dotsc a_s}^i &= \mathbf{p}_i^*\psi_{a_1\dotsc a_s}\in H^0\big( \tilde C^{n+1}, K_{\tilde C}^{(i)}\boxtimes K_{\tilde C}^{(j)}\otimes \hat{\mathfrak{t}}_{g,n} (2\tilde \Delta\!^{(n+1)}) \big),\mbox{ and}\\
		\Psi_{a_1\dotsc a_s}^{i,j} &= \mathbf{p}_{i,j}^*\Psi_{a_1\dotsc a_s} \in H^0\big( \tilde C^{n+1}, K_{\tilde C}^{(i)}\otimes \hat{\mathfrak{t}}_{g,n} (\tilde \Delta\!^{(n+1)}) \big).
	\end{aligned}
	\end{align}
\end{itemize}
\end{definition}

\begin{remark}
Our notation is slightly different from \cite{enriquez}, and their correspondence is as follows:
\begin{align*}
	 \omega_{a_1\dotsc a_s}(p, q) &=  \omega_{a_1\dotsc a_s}^{\underline{p}q},\\
	 \psi_{a_1\dotsc a_s}(p, q)  &= \psi_{a_1\dotsc a_s}^{\underline{p}\underline{q}},\\
	 \Psi_{a_1\dotsc a_s}(p,q_1, q_2) &= \psi_{a_1\dotsc a_s}^{\underline{p}q_1q_2}.
\end{align*}
We use $p$, $p_i$, $q$ etc.\ for points on $C$, and reserve the symbols $z$, $z_i$ and $w$ for local coordinates.
\end{remark}

\begin{definition}
The connection form $\alpha^{(n)}_\mathrm{KZ}$ is defined by
\begin{align*}
	\alpha^{(n)}_\mathrm{KZ} = \sum_{1\leq i\leq n} \alpha_i^{(n)}
\end{align*}
where each $\alpha_i^{(n)}$ is given by
\begin{align*}
	\alpha_i^{(n)} &= \sum_{s\geq 0} \sum_{1\leq a_1,\dotsc a_s, b\leq g} \omega_{a_1\dotsc a_s b}^i \ad \bar x_i^{a_1} \cdots \ad \bar x_i^{a_s}(\bar y_i^b)\\
	&\qquad + \sum_{\substack{1\leq j\leq n\\j\neq i}} \sum_{s\geq 0} \sum_{1\leq a_1,\dotsc a_s \leq g} \Psi^{i,j}_{a_1\dotsc a_s} \ad \bar x_i^{a_1} \cdots \ad \bar x_i^{a_s}(\bar t_{ij}).
\end{align*}

The right-hand side lives in $H^0\big( \tilde C^{n+1}, K_{\tilde C}^{(i)}\otimes \hat{\mathfrak{t}}_{g,n} (\tilde \Delta\!^{(n+1)}) \big)$, but it is shown in \cite{enriquez} that it does not depend on the $(n+1)$-st coordinate of $\tilde C^{n+1}$, so this gives a $1$-form on $\tilde C^n$. Then, $\alpha^{(n)}_\mathrm{KZ}$ is compatible with the monodromy of $\mathcal{P}_n$: for $1\leq j\leq n$, we have
\[
	\gamma_a^{(j)}(\alpha^{(n)}_\mathrm{KZ}) = e^{\ad x^a_j}(\alpha^{(n)}_\mathrm{KZ}),
\]
and therefore descends to $H^0\big( C^n, \Omega^{1,0}_{C^n}\otimes (\mathcal{P}_n\times_\mathrm{ad}\hat{\mathfrak{t}}_{g,n})(\Delta\!^{(n)}) \big)$.
\end{definition}

\begin{lemma}[\cite{enriquez}, Lemma 6]\label{lem:integral}
For $s\geq 1$, $q, q_1, q_2 \in \mathrm{Int}(\mathcal{D})$ and $1\leq a\leq g$, we have
\[
	\int_{p\in \mathcal{A}_a} \omega_{a_1\dotsc a_s}(p, q) = \delta_{aa_1\dotsc a_s} \mathbf{b}_s\mbox{ and } \int_{p\in \mathcal{A}_a} \Psi_{a_1\dotsc a_s}(p, q_1, q_2) = 0,
\]
where $\mathbf{b}_s$ are the coefficients in the series expansion $\sum_{s\geq 0} \mathbf{b}_{s+1} t^s = \frac{t}{e^t - 1}$. 
\end{lemma}

\begin{theorem}[\cite{enriquez}, Theorem 3]\label{thm:enriquez}
The connection $\nabla_n - \alpha^{(n)}_\mathrm{KZ}$ is flat. More specifically, we have $d\alpha^{(n)}_\mathrm{KZ} = [\alpha^{(n)}_\mathrm{KZ}, \alpha^{(n)}_\mathrm{KZ}] = 0$.
\end{theorem}

The case of $g=0$ is not discussed in \cite{enriquez}, but it works without any modification; in fact, the resulting connection is exactly the KZ connection extended to the Riemann sphere.

\begin{remark}
Some choices of an initial data, namely, a choice of
\begin{itemize}
	\item a closed Riemann surface $C$ and a base point $*\in C$,
	\item a marking $\pi_1(C, *) \cong \pi_g$, and
	\item $\tilde *\in \tilde C$ lifting $*\in C$,
\end{itemize}
may result in the same flat connection. Let us fix $(C,*)$ so that we can compare connections on
\[
	\mathcal{P}_n = \tilde C\times_{(F_g)^n} \exp(\hat{\mathfrak{t}}_{g,n}) \to C^n.
\]
Then, we expect that the stabiliser group of $\alpha_\mathrm{KZ}^{(n)}$ is isomorphic to the subgroup
\[
	\{F \in \Aut(\pi_g): F\mbox{ descends to}\, \id\colon F_g\to F_g\}
\]
of $\Aut(\pi_g)$ acting of the set of markings, namely the set of isomorphisms $\mathrm{Isom}(\pi_1(C,*), \pi_g)$.\\
\end{remark}

\section{The Main Construction}\label{sec:main}

From now on, we construct a lift of Enriquez' connection $\alpha_\mathrm{KZ}^{(n)}$ to the framed configuration space.

\begin{definition}
For a finite set $I$, we define the \textit{framed configuration space} $\mathrm{Conf}^\mathrm{fr}_I(C)$ of points labelled by $I$  as the pull-back bundle fitting in the diagram
\[\begin{tikzcd}[cramped]
	\mathrm{Conf}^\mathrm{fr}_I(C) \arrow[r]\arrow[d, "\pi"] &(TC \setminus 0_{TC})^I \arrow[d]\\
	\mathrm{Conf}_I(C) \arrow[r] & C^I
\end{tikzcd}\]
where $TC$ is the tangent bundle of $C$, and $0_{TC}$ is its zero-section. We denote $\pi^{-1}(\Delta^{(I)})$ again by $\Delta^{(I)}$. We write $\mathrm{Conf}^\mathrm{fr}_n(C)$ when $I = \{1,\dotsc, n\}$ as before.
\end{definition}

We need the framed version of the Drinfeld--Kohno Lie algebra.
\begin{definition}
For $g\geq 0$ and a finite set $I$, the Lie algebra $\mathfrak{t}^f_{g,I}$ over $\mathbb{C}$ is generated by elements
\[
	t_{ij}\; (i,j \in I)\quad\mbox{and}\quad x_i^a, y_i^a\; (i \in I, 1\leq a\leq g)
\]
with the relations given, for $i,j,k,l\in I$ and $1\leq a,b\leq g$, by
\begin{align*}
\begin{aligned}
	&t_{ij} = t_{ji},&\\
	&[t_{ij}, t_{kl}] = 0& &\mbox{if } \{i,j\}\cap\{k,l\}=\varnothing,\\
	&[t_{ij},t_{ik} + t_{jk}] = 0& &\mbox{if } \{i,j\}\cap\{k\}=\varnothing,\\
	&[x_i^a, y_j^b] = \delta_{ab} t_{ij}& &\mbox{if } i\neq j,\\
	&[x_i^a, x_j^b] =[y_i^a, y_j^b] = 0& &\mbox{if } i\neq j,\\
	&[x_k^a, t_{ij}] = [y_k^a, t_{ij}] = 0& &\mbox{if } \{i,j\}\cap\{k\}=\varnothing,\\
	&[x_i^a + x_j^a, t_{ij}] = [y_i^a + y_j^a, t_{ij}] = 0
\end{aligned}
\end{align*}
and, for $i\in I$, 
\[
	\sum_{1\leq a\leq g} [x_i^a, y_i^a] + \sum_{j\in I\setminus\{i\}} t_{ij} = (g-1) t_{ii}\,.
\]
\end{definition}
The differences from $\mathfrak{t}_{g,I}$ is the presence of elements $t_{ii}$ $(i\in I)$ and the last relation. We can check that $t_{ii}$ for each $i\in I$ is a central element by substituting $i=j$ in the above relations. In fact, $\mathfrak{t}^f_{g,I}$ is a central extension of $\mathfrak{t}_{g,I}$, forming the exact sequence
\[
	0 \to \bigoplus_{i \in  I}\mathbb{C}t_{ii} \to \mathfrak{t}^f_{g,I} \xrightarrow{q} \mathfrak{t}_{g,I} \to 0\,.
\]
This definition is (equivalent but) a slight modification of the one in \cite{gonzalez}: we changed the coefficient of $t_{ii}$ in the last relation from $2-2g$ to $g-1$ so that the operadic composition map would look simple.\\

For later use, we prepare the following element.

\begin{definition}\label{def:beta}
For $n\geq 0$, we define $\beta_i^{(n)} \in H^0(\tilde C^{n+1}, K_{\tilde C}^{(i)} \otimes \hat{\mathfrak{t}}^f_{g,n}(\tilde \Delta^{(n+1)}))$ by 
\begin{align*}
	\beta_i^{(n)} &= \sum_{s\geq 0} \sum_{1\leq a_1,\dotsc a_s, b\leq g} \omega_{a_1\dotsc a_s b}^i \ad x_i^{a_1} \cdots \ad x_i^{a_s}(y_i^b)\\
	&\qquad + \sum_{\substack{1\leq j\leq n\\j\neq i}} \sum_{s\geq 0} \sum_{1\leq a_1,\dotsc a_s \leq g} \Psi^{i,j}_{a_1\dotsc a_s} \ad x_i^{a_1} \cdots \ad x_i^{a_s}(t_{ij}).
\end{align*}
\end{definition}
Until the end of this section, we fix $n\geq 0$ and drop the index $(n)$. The difference between $\alpha_i$ and $\beta_i$ is that the latter takes the coefficient in $\mathfrak{t}_{g,n}^f$. Contrary to $\alpha_i$, the $1$-form $\beta_i$ does depend on the $(n+1)$-st coordinate of $\tilde C^{n+1}$, and we have the following properties.

\begin{lemma}\label{lem:beta}
For any $1\leq a\leq g$ and $1\leq i,j\leq n$, we have 
\begin{enumerate}[(1)]
	\item $\gamma_a^{(j)}(\beta_i) = e^{\ad x_j^a}(\beta_i)$,
	\item $d^{(n+1)}(\beta_i) = \psi^i \cdot (1-g) t_{ii}$, and
	\item $\int_{p\in \mathcal{A}_a} \beta_i(p_1, \dotsc, p_{i-1}, p, p_{i+1}, \dotsc, p_n, q) = \frac{\ad x_i^a}{e^{\ad x_i^a} - 1} (y_i^a)$
\end{enumerate}
where $p_1, \dotsc, p_n, q\in \mathrm{Int}(\mathcal{D})$.
\end{lemma}
\noindent Proof. (1) It follows from the recursion formulae for $\omega_{a_1\dotsc a_s b}^i$ and $\Psi^{i,j}_{a_1\dotsc a_s}$.\\[3pt]
\noindent (2) Since we have $d^{(n+1)}(\omega_{a_1\dotsc a_s b}^i) = -\delta_{a_s b} \psi_{a_1\dotsc a_{s-1}}^i$ and $d^{(n+1)}(\Psi^{i,j}_{a_1\dotsc a_s}) = -\psi^i_{a_1\dotsc a_s}$, we obtain
\begin{align*}
	&d^{(n+1)}(\beta_i)\\
	&= - \sum_{s\geq 1} \sum_{1\leq a_1,\dotsc a_s, b\leq g} \delta_{a_s b} \psi_{a_1\dotsc a_{s-1}}^i \ad x_i^{a_1} \cdots \ad x_i^{a_s}(y_i^b) - \sum_{j: j\neq i} \sum_{s\geq 0} \sum_{1\leq a_1,\dotsc a_s \leq g} \psi^i_{a_1\dotsc a_s}\ad x_i^{a_1} \cdots \ad x_i^{a_s}(t_{ij})\\
	&= - \sum_{s\geq 1} \sum_{1\leq a_1,\dotsc a_s\leq g} \psi_{a_1\dotsc a_{s-1}}^i \ad x_i^{a_1} \cdots \ad x_i^{a_{s-1}}([x_i^{a_s}, y_i^{a_s}]) - \sum_{j: j\neq i} \sum_{s\geq 0} \sum_{1\leq a_1,\dotsc a_s \leq g} \psi^i_{a_1\dotsc a_s}\ad x_i^{a_1} \cdots \ad x_i^{a_s}(t_{ij})\\
	&= \sum_{s\geq 1} \sum_{1\leq a_1,\dotsc a_{s-1}\leq g} \psi_{a_1\dotsc a_{s-1}}^i \ad x_i^{a_1} \cdots \ad x_i^{a_{s-1}}\Big(\sum_{j:j\neq i}t_{ij} + (1-g)t_{ii} \Big)\\
	&\qquad - \sum_{j: j\neq i} \sum_{s\geq 0} \sum_{1\leq a_1,\dotsc a_s \leq g} \psi^i_{a_1\dotsc a_s}\ad x_i^{a_1} \cdots \ad x_i^{a_s}(t_{ij})\\
	&= \psi^i\cdot (1-g)t_{ii}.
\end{align*}

\noindent (3) This follows from Lemma \ref{lem:integral} together with the definition of $\beta_i$. \qed\\

Now take a parametrisation $\ell\colon S^1 \to \mathcal{A}_a$ of the simple loop $\mathcal{A}_a$. We define the loop $\vec{\mathcal{A}}_a$ in $T\tilde C\setminus 0_{T\tilde C}$ as the image of
\begin{align}\label{eq:vecA}
	S^1 \to T\tilde C\setminus 0_{T\tilde C}: t \mapsto (\ell(t), \frac{d\ell}{dt}(t)).
\end{align}
Then, $\vec{\mathcal{A}}_a$ is clearly a lift of $\mathcal{A}_a$ along the projection $T\tilde C\setminus 0_{T\tilde C} \to \tilde C$, and its homotopy class in $T\tilde C\setminus 0_{T\tilde C}$ does not depend on the parametrisation. We also denote its image by $d\varpi\colon T\tilde C \to TC$ by $\vec{\mathcal{A}}_a$.

For a local coordinate $z\colon C\supset U \to \mathbb{C}$, let
\begin{align}\label{eq:assoccoord}
	(z, \lambda) \colon TU \xrightarrow{dz} T\mathbb{C} \cong \mathbb{C}^2
\end{align}
be the associated local cordinate of $TC$. Let $\Delta'$ be the inverse image of the diagonal in $C\times C$ along $TC\times \tilde C \to C\times C$.

\begin{lemma}\label{lem:phi}
There uniquely exists a meromorphic $1$-form
\[
	\varphi \in H^0(TC \times \tilde C, (\Omega^{1,0}_{TC}\boxtimes \mathcal{O}_{\tilde C})(\Delta' + 0_{TC}))
\]
with the following properties:
\begin{enumerate}[(1)]
	\item For each local coordinate $(z,\lambda)$ of $TC$ as above, $\varphi$ is locally of the form
	\[	
		\frac{d\log\lambda}{2\pi\sqrt{-1}}  + H^0\big( C\times \tilde C, (K_C \boxtimes \mathcal{O}_{\tilde C})(\Delta') \big)
	\]
	and hence $d^{(1)}\varphi = 0$,
	\item $\mathrm{res}_{\Delta'} \varphi = \frac{2g-2}{2\pi\sqrt{-1}}$,
	\item $d^{(2)}(\varphi) = (2g-2)\psi$, and
	\item $\int_{\vec p\in \vec{\mathcal{A}}_a} \varphi(\vec p, q) = 0$ for $q\in\mathrm{Int}(\mathcal{D})$.
\end{enumerate}
Here, $d^{(1)}$ (resp.\ $d^{(2)}$) is the exterior derivative applied to the first factor $TC$ (resp.\ the second factor $\tilde C$) of $TC\times \tilde C$.
\end{lemma}
\noindent Proof. (Uniqueness) Let $\varphi, \varphi'$ be two such $1$-forms and put $\Phi = \varphi - \varphi'$. Then, the condition (1) implies that $\Phi$ descends to $C\times \tilde C$ since the term $\frac{d\log\lambda}{2\pi\sqrt{-1}}$ cancels out  on any local coordinate as above. Moreover, (2) forces $\Phi$ to be holomorphic. We also have
\[
	d^{(2)}(\Phi) = 0\quad\mbox{and}\quad \int_{\vec p\in \vec{\mathcal{A}}_a} \Phi(\vec p, q) = 0,
\]
which amounts to say that $\Phi$ is a holomorphic $1$-form on $C$ with all the integration over $A$-loops are zero. Thus, we have $\Phi = 0$.

\noindent (Existence) Let $\omega$ be any non-zero meromorphic $1$-form on $C$ and consider the following meromorphic function
\begin{align*}
	F_\omega\colon TC& \to \mathbb{C}\\
	\vec p = (p,v) &\mapsto \langle \omega, v\rangle
\end{align*}
on $TC$, where $p\in C$ and $v\in T_pC$. This defines a meromorphic $1$-form $d\log F_\omega$ on $TC$. For a local coordinate $(z, \lambda)$ of $TC$ as above, putting $\omega = f dz$ with some meromorphic function $f$, we have
\[
	F_\omega = f \lambda
\]
and therefore
\[
	d\log F_\omega = d\log \lambda + d\log f.
\]
Since the divisor $(\omega) = \sum_i n_i c_i \,(n_i \in \mathbb{Z}, c_i\in C)$ has degree $2g-2$, $d\log F_\omega$ has a simple pole at $c_i$ with the residue $n_i$ such that $\sum_i n_i = 2g-2$, and along $0_{TC}$ with the residue $1$. 

On the other hand, $\Psi \in H^0\big( \tilde C^3, K_{\tilde C} ^{(1)}(\tilde\Delta\!^{(3)}) \big)$ is $F_g$-invariant under the action to the first component of $\tilde C^3$, so we regard $\Psi$ as a $1$-form on $C\times \tilde C^2$. Then, it has simple poles only on the loci
\[
	\{(p,q_1,q_2)\in C\times \tilde C^2: p = \varpi(q_1)\}\quad\mbox{and}\quad\{(p,q_1,q_2)\in C\times \tilde C^2: p = \varpi(q_2)\} 
\]
with the residues $\frac{-1}{2\pi\sqrt{-1}}$ and $\frac{1}{2\pi\sqrt{-1}}$, respectively, by the expansion \eqref{eq:bidiff}. We choose an arbitrary lift $\tilde c_i \in \tilde C$ of $c_i$ and consider the $1$-form
\[
	\frac{d\log F_\omega(\vec p) }{2\pi\sqrt{-1}}+ \sum_i n_i \Psi(p, \tilde c_i, q)
\]
on $TC\times \tilde C$, whose poles are only on $\{(\vec p, q)\in TC\times \tilde C: p =  \varpi(q)\}$ with the residue $\frac{2g-2}{2\pi\sqrt{-1}}$. Now take any $q_0\in\mathrm{Int}(\mathcal{D})$ and a (unique) holomorphic $1$-form $h$ on $C$ so that
\[
	\varphi(\vec p, q) := \frac{d\log F_\omega(\vec p) }{2\pi\sqrt{-1}}+ \sum_i n_i \Psi(p, \tilde c_i, q) + h(p)
\]
are normalised by $\{A_a\}_{1\leq a\leq g}$ at $q_0$:
\[
	\int_{\vec p\in\vec{\mathcal{A}}_a} \varphi(\vec p, q_0) = 0
\]
for all $1\leq a\leq g$.

Now we check (1)--(4). First of all, (1) and (2) follow from the above expression of $\varphi$. (3) follows from $d^{(3)}(\Psi)(p, q_1, q_2) = \psi(p,q_2)$. For (4), we use (3) to compute
\[
	d\int_{\vec p\in \vec{\mathcal{A}}_a} \varphi(\vec p, q) = \int_{p\in A_a} (2g-2)\psi(p, q) = 0,
\]
which implies the integral
\[
	\int_{\vec p\in \vec{\mathcal{A}}_a} \varphi(\vec p, q)
\]
is independent of $q\in\mathrm{Int}(\mathcal{D})$. Since this vanishes at $q = q_0$, the claim (4) follows. \qed\\[-7pt]

\begin{definition}\label{def:conn}
We define the genus $g$ \textit{framed KZB connection form}
\[	
	\vec\alpha_\mathrm{KZ} \in  H^0(T\tilde C^n\times \tilde C, (\Omega^{1,0}_{T\tilde C^n}\boxtimes \mathcal{O}_{\tilde C})\otimes \hat{\mathfrak{t}}^f_{g,n} (\tilde \Delta^{(n+1)} + \sum_{1\leq i\leq n}0_{T\tilde C}^{(i)}))
\]
by the formula
\[
	\vec\alpha_\mathrm{KZ} = \sum_{1\leq i\leq n} \Big(\beta_i + \varphi^i \frac{t_{ii}}{2} \Big),
\]
where $0_{T\tilde C}^{(i)}$ is the $i$-th copy of $0_{T\tilde C}$, $\beta_i$ is regarded as the $1$-form on $T\tilde C^n\times \tilde C$ via the pull-back by $T\tilde C^n\times \tilde C \to \tilde C^{n+1}$, and $\varphi^i$ is the pull-back of $\varphi$ along
\[
	T\tilde C^n\times \tilde C \to TC \times \tilde C: (\vec p_1,\dotsc, \vec p_n, q) \mapsto (d\varpi(\vec p_i), q).
\]
\end{definition}

\begin{theorem}\label{thm:kzb}
Let $q\colon \mathfrak{t}_{g,n}^f \twoheadrightarrow \mathfrak{t}_{g,n}$ and $\pi\colon \mathrm{Conf}_n^\mathrm{fr}(C) \twoheadrightarrow \mathrm{Conf}_n(C)$ be the quotient maps. Then,
\begin{enumerate}[(1)]
	\item $\vec\alpha_\mathrm{KZ}$ induces a meromorphic $\exp(\hat{\mathfrak{t}}^f_{g,n})$-connection on $\mathcal{P}^f_n \to TC^n$, where $\mathcal{P}^f_n$ is defined by
	\[
		\mathcal{P}^f_n = T\tilde C^n \times_{(F_g)^n} \exp(\hat{\mathfrak{t}}^f_{g,n})
	\]
	with $(F_g)^n$ acting on $T\tilde C^n$ via deck transformations and on $\exp(\hat{\mathfrak{t}}^f_{g,n})$ by
	\begin{align*}
		(F_g)^n &\to \exp(\hat{\mathfrak{t}}^f_{g,n})\\
		\gamma_a^{(i)} &\mapsto e^{x_i^a};
	\end{align*}
	\item $\vec\alpha_\mathrm{KZ}$ restricts to a holomorphic connection on $\mathrm{Conf}_n^\mathrm{fr}(C)$;
	\item $q(\vec\alpha_\mathrm{KZ}) = \pi^*(\alpha_\mathrm{KZ})$;
	\item the connection induced by $\vec\alpha_\mathrm{KZ}$ on $\mathrm{Conf}_n^\mathrm{fr}(C)$ is flat; and
	\item the restriction to a fibre of $\pi$ is $\sum_{1\leq i\leq n}  \frac{d\log\lambda_i}{4\pi\sqrt{-1}} t_{ii}$.
\end{enumerate}
\end{theorem}
\noindent Proof. (1) We have to show that $d^{(n+1)}(\vec\alpha_\mathrm{KZ}) = 0$ and $\gamma_a^{(j)}(\vec\alpha_\mathrm{KZ}) = e^{\ad x_j^a} (\vec\alpha_\mathrm{KZ})$. The former follows from Lemmata \ref{lem:beta} (2) and \ref{lem:phi} (3). The latter follows from Lemma \ref{lem:beta} (1), the fact that $\varphi$ is $F_g$-invariant on the first coordinate, and that $t_{ii}$ is central in $\mathfrak{t}_{g,n}^f$.

\noindent (2) Regarding $\vec\alpha_\mathrm{KZ}$ as a multi-valued $1$-form on $TC^n$, the poles are contained in $\Delta^{(n)} + \sum_{1\leq i\leq n}0_{TC}^{(i)}$, which is exactly the complement of $\mathrm{Conf}_n^\mathrm{fr}(C)$ in $TC^n$. Therefore, $\vec\alpha_\mathrm{KZ}$ is holomorphic on $\mathrm{Conf}_n^\mathrm{fr}(C)$.

\noindent (3) We have $q(\vec\alpha_\mathrm{KZ}) = \sum_{1\leq i\leq n} q(\beta_i)$, but each $q(\beta_i)$ is equal to $\pi^*(\alpha_i)$ by its definition.

\noindent (4) We can show that $d^{(i)}\beta_j = d^{(j)} \beta_i$ in the same manner with Lemma 15 in \cite{enriquez} since the only relations used in the proof are $[x_i^a + x_j^a, t_{ij}] = 0$ and $[x_i^a, x_j^b] = 0$. In addition, we have $d^{(1)} \varphi= 0$ by the condition (1) in Lemma \ref{lem:phi}. This shows $d\vec\alpha_\mathrm{KZ} = 0$.

By (3) and Theorem \ref{thm:enriquez}, we have $q([\vec\alpha_\mathrm{KZ}, \vec\alpha_\mathrm{KZ}]) = \pi^*([\alpha_\mathrm{KZ}, \alpha_\mathrm{KZ}]) = 0$. This implies that $[\vec\alpha_\mathrm{KZ}, \vec\alpha_\mathrm{KZ}]$ lies in
\[
	\Ker(q\colon \mathfrak{t}^f_{g,n} \to \mathfrak{t}_{g,n}) \cong \bigoplus_{1\leq i\leq n}\mathbb{C} t_{ii}.
\]
With the grading $\deg(x^a_i) = 0$, $\deg(y^a_i) = 1$ and $\deg(t_{ij}) = 1$ on $\mathfrak{t}^f_{g,n}$, we have $\deg(\vec\alpha_\mathrm{KZ}) = 1$ and therefore $\deg([\vec\alpha_\mathrm{KZ}, \vec\alpha_\mathrm{KZ}]) = 2$. On the other hand, $\Ker q$ is contained in degree-$1$ part, so we obtain $[\vec\alpha_\mathrm{KZ}, \vec\alpha_\mathrm{KZ}] = 0$. This shows the flatness $d\vec\alpha_\mathrm{KZ} = [\vec\alpha_\mathrm{KZ}, \vec\alpha_\mathrm{KZ}] = 0$.

\noindent (5) This follows from Lemma \ref{lem:phi} (1).\qed

\begin{remark}
The above theorem does \textit{not} characterize $\vec\alpha_\mathrm{KZ}$. For example, adding any closed holomorphic $1$-form on $C^n$ with coefficients in $\bigoplus_{1\leq i\leq n}\mathbb{C}t_{ii}$ will also result in another connection satisfying (1)--(4) above. However, $\vec\alpha_\mathrm{KZ}$ is the canonical choice in the sense that it is obtained as a degeneration of the original flat connection $\alpha_\mathrm{KZ}$ and therefore \textit{operadic}; see Theorem \ref{thm:operadic}.\\
\end{remark}

\section{Degeneration of Connections: Operadicity}\label{sec:operadicity}

As noted in the introduction, the most notable feature of the universal KZ connection is the compatibility with the operadic structure of the configuration space $\mathrm{Conf}_n(\mathbb{C})$. In this section, we show that $\vec\alpha_\mathrm{KZ}^{(n)}$ also has this property.\\

We fix a Riemannian metric on $C$, and endow the covering space $\tilde C$ with the pull-back metric. Then, the action of the group of deck transformations $\Aut(\tilde C/C) \cong F_g$ on $\tilde C$ preserves the metric. Since $C$ is compact and therefore geodesically complete by the Hopf--Rinow Theorem, $\tilde C$ is also geodesically complete under the pull-back metric.

\begin{definition}
We denote the exponential map on $\tilde C$ associated with the above metric by
\[
	\mathrm{Exp}\colon T\tilde C \to \tilde C: (p,v) \mapsto \mathrm{Exp}_p(v),
\]
where $p\in \tilde C$ and $v\in T_p\tilde C$. By definition, the map
\[
	\mathbb{R}\to \tilde C\colon t\mapsto \mathrm{Exp}_p(tv)
\]
is a unique geodesic on $\tilde C$ with $\mathrm{Exp}_p(0) = p$ such that the derivative at $0\in T_p\tilde C$,
\[
	(d\mathrm{Exp}_p)_0\colon T_0T_p\tilde C \to T_p\tilde C,
\]
is the identity map via the canonical identification $T_p\tilde C \cong T_0T_p\tilde C$.
\end{definition}
Let $p\in \tilde C$, $z$ a local coordinate on $\tilde C$ near $p$, and $(z,\lambda)$ the associated coordinate of $T\tilde C$ obtained as in \eqref{eq:assoccoord}. For  a small complex parameter $\varepsilon$, we have the expansion
\begin{align}\label{eq:expansion}
	z(\mathrm{Exp}_p(\varepsilon v)) = z(p) + \varepsilon\lambda(p, v) + O(\varepsilon^2).
\end{align}
The choice of a metric only affects the $O(\varepsilon^2)$-term, and we will see below that $O(\varepsilon^2)$-term does not contribute to our calculation in the limit of $\varepsilon \to 0$.

We now move on to the operadic structure on $\mathrm{Conf}_n^\mathrm{fr}(C)$. The following map corresponds to doubling the strand labelled by $n$ on braid groups.

\begin{definition}\label{def:fvarepsilon}
Let $n\geq 0$, and $B$ a small disk in $\mathbb{C}$ centered at $0$. For $\varepsilon\in B$, we put
\begin{align*}
	f_\varepsilon\colon T\tilde C^n &\to T\tilde C^{n+1}\\
	(p_1, v_1;\dotsc;p_n, v_n)&\mapsto \big( p_1,v_1; \dotsc; p_n, v_n;\mathrm{Exp}_{p_n}(\varepsilon v_n), (d\mathrm{Exp}_{p_n})_{\varepsilon v_n}(v_n) \big),
\end{align*}
where $p_i\in \tilde C$ and $v_i\in T_{p_i}\tilde C$. 
\end{definition}

We denote by $(z_1,\lambda_1;\dotsc;z_n,\lambda_n)$ the local coordinate of $T\tilde C^n$ consisting of $n$ copies of $(z,\lambda)$ above. Then, using the expansion \eqref{eq:expansion}, we have
\begin{align*}
	f_\varepsilon^*(z_i) = z_i,\quad f_\varepsilon^*(\lambda_i) = \lambda_i,\quad f_\varepsilon^*(z_{n+1}) = z_n + \varepsilon\lambda_n + O(\varepsilon^2), \quad f_\varepsilon^*(\lambda_{n+1}) = \lambda_n + O(\varepsilon)
\end{align*}
for $1\leq i \leq n$.

\begin{lemma}
Let $(F_g)^n$ act on $\tilde C^{n+1}$ via the map
\begin{align*}
	(F_g)^n \to (F_g)^{n+1}: \gamma_a^{(i)} \mapsto \gamma_a^{(i)}\,(1\leq i<n),\quad \gamma_a^{(n)} \mapsto \gamma_a^{(n)}\gamma_a^{(n+1)}
\end{align*}
between the group of deck transformations. Then, the map $f_\varepsilon$ is $(F_g)^n$-equivariant.
\end{lemma}
\noindent Proof. Since $F_g$ acts on $\tilde C$ preserving the metric, we have
\[
	\gamma(\mathrm{Exp}_p(v)) = \mathrm{Exp}_{\gamma(p)}((d\gamma)_p(v))
\]
for any $\gamma \in F_g$, $p\in \tilde C$ and $v\in T_p\tilde C$. Now the claim follows from the definition of $f_\varepsilon$.\qed\\

We have the morphism of Lie algebras $(-)^{1,\dotsc, n-1, n(n+1)}\colon \mathfrak{t}_{g,n}^f \to \mathfrak{t}_{g,n+1}^f$ induced by the operadic composition map on a family $\{\mathfrak{t}_{g,n}^f\}_{n\geq 0}$. This is described by
\begin{align*}
	t_{ij} &\mapsto t_{ij}\\
	t_{in} &\mapsto t_{in} + t_{i,n+1}\\
	t_{nn} &\mapsto t_{nn} + 2t_{n,n+1} + t_{n+1, n+1}\\
	x_i^a &\mapsto x_i^a\\
	x_n^a &\mapsto x_n^a + x_{n+1}^a\\
	y_i^a &\mapsto y_i^a\\
	y_n^a &\mapsto y_n^a + y_{n+1}^a
\end{align*}
for $1\leq i,j < n$ and $1\leq a\leq g$. Recall that the operadic composition map on $\{\mathfrak{t}_{g,n}^f\}_{n\geq 0}$ was induced by that of braids on a genus $g$ surface, and the above map $(-)^{1,\dotsc, n-1, n(n+1)}$ corresponds to the doubling of the $n$-th strand described by the map $f_\varepsilon$.

Then, the operadicity of $\vec\alpha_\mathrm{KZ}^{(n)}$ is stated as follows. Recall that $\vec{\alpha}_\mathrm{KZ}^{(n)}$ is a meromorphic $1$-form on $T\tilde C^n$ by Theorem \ref{thm:kzb} (1).

\begin{theorem}\label{thm:conv}
Let $n\geq 0$ and $\mathrm{S}_n$ the symmetric group of degree $n$. Furthermore, let $B$ be a small disk in $\mathbb{C}$ centered at $0$ as in Definition \ref{def:fvarepsilon}, and put $B^\times = B\setminus \{0\}$.
\begin{enumerate}[(1)]
	\item The $1$-form $\vec{\alpha}_\mathrm{KZ}^{(n)}$ is invariant under the action of $\mathrm{S}_n$ on $T\tilde C^n$ by the permutation of points.
	\item The family of $1$-forms $\{f_\varepsilon^*(\vec{\alpha}_\mathrm{KZ}^{(n+1)})\}_{\varepsilon\in B^\times}$ on $T\tilde C^n$ converges to $(\vec{\alpha}_\mathrm{KZ}^{(n)})^{1,\dotsc, n-1, n(n+1)}$ as $\varepsilon \to 0$.
\end{enumerate}
\end{theorem}

We note that this is a refinement of the formula for the simplicial behaviour (Proposition 13 of \cite{enriquez}). Before we proceed to the proof, we have to clarify in what sense this family of $1$-forms converges.
\begin{definition}
Let $X$ be a complex manifold, and $\{(\omega_\varepsilon, S_\varepsilon)\}_{\varepsilon \in B^\times}$ a family, where $S_\varepsilon$ is a subset of $X$, and $\omega_\varepsilon$ is a $C^\infty$ $1$-form defined on $X\setminus S_\varepsilon$ with coefficients in $\mathbb{C}$ for each $\varepsilon \in B^\times$.
\begin{itemize}
	\item Let $S$ be a subset of $X$. We say that the family $\{(\omega_\varepsilon, S_\varepsilon)\}_{\varepsilon \in B^\times}$ \textit{converges outside} $S$ as $\varepsilon \to 0$, if, for each $p\in X\setminus S$ and any compact neighbourhood $U\subset X\setminus S$ of $p$, there exist:
	\begin{enumerate}[(1)]
		\item $r>0$ such that $U \cap S_\varepsilon = \varnothing$ for $|\varepsilon| < r$, and
		\item (at least) continuous $1$-forms $\omega_{0, U}$ and $\sigma_{\varepsilon, U}$ on $U$ with the latter continuously depend on $\varepsilon \in B$, such that
		\[
			\omega_\varepsilon|_U = \omega_{0, U} + \varepsilon\sigma_{\varepsilon, U}
		\]
		for any $\varepsilon \in B^\times$ with $|\varepsilon| < r$.
	\end{enumerate}
	Such $\omega_{0, U}$ is unique, and it defines a $1$-form on $X\setminus S$, which we denote by $\lim_{\varepsilon \to 0} \omega_\varepsilon$. We also use the symbol $O(\varepsilon)$ for the term $\varepsilon\sigma_{\varepsilon, U}$.
	\item Let $\omega$ be a meromorphic $1$-form on $X$, and $\mathrm{sing}(\omega)$ its singular locus. We say that the above family $\{(\omega_\varepsilon, S_\varepsilon)\}_{\varepsilon \in B^\times}$ \textit{converges to} $\omega$ as $\varepsilon \to 0$, if $\{(\omega_\varepsilon, S_\varepsilon)\}_{\varepsilon \in B^\times}$ converges outside $\mathrm{sing}(\omega)$ and, moreover, $\lim_{\varepsilon \to 0} \omega_\varepsilon = \omega$ holds pointwisely on $X\setminus \mathrm{sing}(\omega)$.
\end{itemize}
Let $\mathfrak{g}$ be a graded Lie algebra over $\mathbb{C}$, and suppose that the degree-$d$ part $\mathfrak{g}_d$ is finite-dimensional for each $d$.
\begin{itemize}
	\item We say that a family $\{(\omega'_\varepsilon, S_\varepsilon)\}_{\varepsilon \in B^\times}$ of above $1$-forms on $X$ but with coefficients in $\mathfrak g$ converges (either outside $S\subset X$ or to some $\mathfrak{g}$-valued meromorphic $1$-form $\omega'$), if, for any $d$, in the degree-$d$ part of $\{(\omega_\varepsilon, S_\varepsilon)\}_{\varepsilon \in B^\times}$, each coeffieicient with respect to some basis of $\mathfrak{g}_d$ converges in the above sense.\\
\end{itemize}
\end{definition}

We shall do some preparation for the proof of Theorem \ref{thm:conv}. We denote
\[
	f_\varepsilon\times \id_{\tilde C}\colon T\tilde C^n \times \tilde C\to T\tilde C^{n+1} \times \tilde C
\]
again by $f_\varepsilon$. Note that $f_0$ is a holomorphic map. Let $\omega$ be a meromorphic $1$-form on $T\tilde C^{n+1} \times \tilde C$. Whenever we consider the family $\{f^*_\varepsilon(\omega)\}_{\varepsilon \in B^\times}$ on $T\tilde C^n \times \tilde C$, we will always set $S_\varepsilon = f_\varepsilon^{-1}(\mathrm{sing}(\omega))$ and drop $S_\varepsilon$ from the notation.

\begin{lemma}\label{lem:holconv}
Let $\omega$ be a meromorphic $1$-form on $T\tilde C^{n+1} \times \tilde C$, and $p\in T \tilde C^n \times \tilde C$. If $f_0(p)\notin \mathrm{sing}(\omega)$, then $\{f^*_\varepsilon(\omega)\}_{\varepsilon \in B^\times}$ converges to the meromorphic $1$-form $f_0^*(\omega)$ in some neighbourhood $U$ of $p$.
\end{lemma}
\noindent Proof. Consider the continuous map
\begin{align*}
	f\colon B\times T\tilde C^n \times \tilde C &\to T\tilde C^{n+1} \times \tilde C\\
	(\varepsilon, p) &\mapsto  f_\varepsilon(p).
\end{align*}
Since $\mathrm{sing}(\omega) \subset T\tilde C^{n+1} \times \tilde C$ is closed, $f^{-1}((T\tilde C^{n+1} \times \tilde C)\setminus \mathrm{sing}(\omega))$ is open so that we can take a neighbourhood $\{\varepsilon\in B \colon |\varepsilon| < r\} \times U$ of $(0, p)$ in $B\times (T\tilde C^n \times \tilde C)$ for some $r > 0$ and $U$ compact. Then, we have $U \cap S_\varepsilon = \varnothing$ for $|\varepsilon| < r$, and $f_\varepsilon^*(\omega)$ is of class $C^\infty$ on $U$. Then, we have
\begin{align*}
	f_\varepsilon^*(\omega) = f_0^*(\omega) + O(\varepsilon)
\end{align*}
on $U$ using the expansion \eqref{eq:expansion}, where $O(\varepsilon)$-term is estimated in terms of first derivatives of $f_0^*(\omega)$; this was our definition of the convergence. \qed\\

Denote by $\Delta_{i,j}^{(n)}$ the $(i,j)$-diagonal of $C^n$, and put $\tilde \Delta_{i,j}^{(n)} = (\varpi^n)^{-1}(\Delta_{i,j}^{(n)})$ using the projection $\varpi^n\colon \tilde C^n \to C^n$.

\begin{lemma}\label{lem:degen}
Let $\beta_i^{(n)}$ be the meromorphic $1$-form on $\tilde C^{n+1}$ in Definition \ref{def:beta}, which we naturally regard to be defined on $T\tilde C^n \times \tilde C$ via the pull-back by $T\tilde C^n \times \tilde C \to \tilde C^{n+1}$. 
\begin{enumerate}[(1)]
	\item Let $D = 0_{T\tilde C}^{(n)} + \sum_{1\leq i< n}\tilde \Delta^{(n+1)}_{i,n} + \sum_{1\leq i\leq n}\tilde \Delta^{(n+1)}_{i,n+1}$ be a divisor of $T\tilde C^n \times\tilde C$. Then, the family of $1$-forms
	\[
		\{f_\varepsilon^*(\beta_n^{(n+1)} + \beta_{n+1}^{(n+1)})\}_{\varepsilon\in B^\times}
	\]
	on $T\tilde C^n \times\tilde C$ converges outside $D$ as $\varepsilon \to 0$. We denote the limit by $\eta^{(n)}$.
	\item $\eta^{(n)}$ is, moreover, a meromorphic $1$-form on $T \tilde C^n \times\tilde C$ with the divisor $D$:
	\[
		\eta^{(n)} \in H^0\Big(T \tilde C^n \times \tilde C, ((\Omega_{T\tilde C}^{1,0})^{(n)}\boxtimes \mathcal{O}_{\tilde C})\otimes \hat{\mathfrak{t}}_{g,n}^f (D)\Big).
	\]
	\item $\mathrm{res}_{p_i = p_n}\eta^{(n)}(\vec p_1,\dotsc, \vec p_n; q) = \frac{t_{in} + t_{i,n+1}}{2\pi\sqrt{-1}}$ for $1\leq i < n$, and $\mathrm{res}_{0^{(n)}_{T\tilde C}}\eta^{(n)} = \frac{t_{n,n+1}}{2\pi\sqrt{-1}}$.
	\item $\gamma_a^{(i)}(\eta^{(n)}) = (e^{\ad x^a_i})^{1,\dotsc, n-1, n(n+1)}(\eta^{(n)})$ for any $1\leq i\leq n$ and $1\leq a\leq g$.
	\item $d^{(n+1)}\eta^{(n)} = \psi^n\cdot (1-g)(t_{nn} + t_{n+1, n+1})$.
\end{enumerate} 
\end{lemma}
\noindent Proof. (1) By the definition of $\beta_i^{(n+1)}$, we have
\begin{align*}
	f_\varepsilon^*(\beta_n^{(n+1)} + \beta_{n+1}^{(n+1)}) &= \sum_{s\geq 0} \sum_{1\leq a_1,\dotsc a_s, b\leq g} f_\varepsilon^*(\omega_{a_1\dotsc a_s b}^n) \ad x_n^{a_1} \cdots \ad x_n^{a_s}(y_n^b)\\
	&\qquad + \sum_{\substack{1\leq j\leq n+1\\j\neq n}} \sum_{s\geq 0} \sum_{1\leq a_1,\dotsc a_s \leq g} f_\varepsilon^*(\Psi^{n,j}_{a_1\dotsc a_s}) \ad x_n^{a_1} \cdots \ad x_n^{a_s}(t_{nj})\\
	&\quad + \sum_{s\geq 0} \sum_{1\leq a_1,\dotsc a_s, b\leq g} f_\varepsilon^*(\omega_{a_1\dotsc a_s b}^{n+1}) \ad x_{n+1}^{a_1} \cdots \ad x_{n+1}^{a_s}(y_{n+1}^b)\\
	&\qquad + \sum_{\substack{1\leq j\leq n+1\\j\neq {n+1}}} \sum_{s\geq 0} \sum_{1\leq a_1,\dotsc a_s \leq g} f_\varepsilon^*(\Psi^{{n+1},j}_{a_1\dotsc a_s}) \ad x_{n+1}^{a_1} \cdots \ad x_{n+1}^{a_s}(t_{{n+1},j}).
\end{align*}
We examine the convergence outside $D$, so take $p = ((\vec p_1,\dotsc, \vec p_n; q)\in (T\tilde C^n \times \tilde C)\setminus D$ and put $\vec p_i = (p_i, v_i)$. We first check to which terms we can apply Lemma \ref{lem:holconv}.
\begin{itemize}
	\item As for the first term, we have $\mathrm{sing}(\omega_{a_1\dotsc a_s b}^n) = \tilde\Delta^{(n+2)}_{n, n+2} \subset T\tilde C^{n+1}\times\tilde C$ by its definition \eqref{eq:omegadef}. Suppose that
	\[
		f_0(p) = (\vec p_1,\dotsc, \vec p_n, \vec p_n ;q) \in \tilde\Delta^{(n+2)}_{n, n+2}
	\]
	holds. Then, we have $(p_n, q) \in \tilde\Delta^{(2)}_{1, 2}$ and hence $p\in \tilde\Delta^{(n+1)}_{n, n+1} \subset D$, which contradicts to $p\notin D$. We conclude $f_0(p)\notin \mathrm{sing}(\omega_{a_1\dotsc a_s b}^n)$. By Lemma \ref{lem:holconv}, we have $f_\varepsilon^*(\omega_{a_1\dotsc a_s b}^n)\to f_0^*(\omega_{a_1\dotsc a_s b}^n)$ near $p$.
	\item Similarly, as for the third term, we have $\mathrm{sing}(\omega_{a_1\dotsc a_s b}^{n+1}) = \tilde\Delta^{(n+2)}_{n+1, n+2} \subset T\tilde C^{n+1}\times\tilde C$. If
	\[
		f_0(p) = (\vec p_1,\dotsc, \vec p_n, \vec p_n ;q) \in \tilde\Delta^{(n+2)}_{n+1, n+2}
	\]
	holds, then, we have $(p_n, q) \in \tilde\Delta^{(2)}_{1, 2}$ and hence $p\in \tilde\Delta^{(n+1)}_{n, n+1} \subset D$, which contradicts to $p\notin D$. We conclude $f_0(p)\notin \mathrm{sing}(\omega_{a_1\dotsc a_s b}^{n+1})$ and therefore $f_\varepsilon^*(\omega_{a_1\dotsc a_s b}^{n+1})\to f_0^*(\omega_{a_1\dotsc a_s b}^{n+1})$ near $p$.
	\item As for the second term, we first consider the case $1\leq j<n$. Suppose $f_0(p)\in \mathrm{sing}(\Psi^{n,j}_{a_1\dotsc a_s})$, which is further contained in
	\[
		\tilde\Delta^{(n+2)}_{j, n+2} + \tilde\Delta^{(n+2)}_{j, n} + \tilde\Delta^{(n+2)}_{n, n+2}
	\]
	by \eqref{eq:omegadef}. Similar to the above, assuming that $f_0(p)$ is in any of these three yields the contradiction to $p\notin D$. In this case, we have $f_\varepsilon^*(\Psi^{n,j}_{a_1\dotsc a_s}) \to f_0^*(\Psi^{n,j}_{a_1\dotsc a_s})$ near $p$ as before.
	
	\quad Now consider the case $j=n+1$, and suppose
	\[
		f_0(p) = (\vec p_1,\dotsc, \vec p_n, \vec p_n ;q)\in \tilde\Delta^{(n+2)}_{n+1, n+2} + \tilde\Delta^{(n+2)}_{n+1, n} + \tilde\Delta^{(n+2)}_{n, n+2}.
	\]
	We can reject the cases $f_0(p) \in \tilde\Delta^{(n+2)}_{n, n+2}$ and $f_0(p) \in \tilde\Delta^{(n+2)}_{n+1, n+2}$ as above, so we assume $f_0(p) \in \tilde\Delta^{(n+2)}_{n+1, n}$, which amounts to say that $(p_n, q, p_n)\in \tilde\Delta^{(3)}_{1, 2}$ is contained in the singular locus of $\Psi_{a_1\dotsc a_s}$ by the definition of the projection $\mathbf{p}_{n,n+1}$.
	If $s\geq 1$, the residue condition \eqref{eq:residue} together with the formula $\psi_{a_1\dotsc a_s} = - d^{(2)}(\omega_{a_1\dotsc a_s b b})$ says that $\Psi_{a_1\dotsc a_s}$ is holomorphic near
	\[
		\{(q_0,q_1,q_2) \in \tilde C^3: q_0 = q_2, \varpi(q_1)\neq \varpi(q_2) \},
	\]
	which contains $(p_n, q, p_n)$. This is a contradiction, so we conlude that $f_0(p)\notin \mathrm{sing}(\Psi_{a_1\dotsc a_s})$. We have $f_\varepsilon^*(\Psi^{n,n+1}_{a_1\dotsc a_s}) \to f_0^*(\Psi^{n,n+1}_{a_1\dotsc a_s})$ near $p$.
	
	\quad The remaining case of $(j, s) = (n+1, 0)$ with $(p_n, q, p_n)\in \tilde\Delta^{(3)}_{1, 2}$ is separately dealt with later.
	\item Similarly, as for the fourth term, we have $f_\varepsilon^*(\Psi^{n+1,j}_{a_1\dotsc a_s}) \to f_0^*(\Psi^{n+1,j}_{a_1\dotsc a_s})$ near $p$, except for $(j,s) = (n,0)$ with $(p_n, q, p_n)\in \tilde\Delta^{(3)}_{1, 2}$.
\end{itemize}

The remaining two terms are precisely the contribution of 
\begin{align*}
	f_\varepsilon^*(\Psi^{n,n+1}) t_{n, n+1} +  f_\varepsilon^*(\Psi^{{n+1},n}) t_{{n+1},n}.
\end{align*}
We check the convergence on a local coordinate. Let $z_i$ be a local coordinate of $\tilde C$ near $p_i$, $w$ near $q$, and $(z_i,\lambda_i)$ the associated local coordinate of $T\tilde C$ with $z_i$ as before. Then,
\[	
	(z_1,\lambda_1;\dotsc;z_{n},\lambda_{n}; w)
\]
is a local coordinate of $T\tilde C^{n}\times \tilde C$ near $p = (\vec p_1,\dotsc, \vec p_n, \vec p_n ;q)$, and
\[
	(z_1,\lambda_1;\dotsc;z_{n},\lambda_{n};z_{n},\lambda_{n}; w)
\]
is a local coordinate of $T\tilde C^{n+1}\times \tilde C$ near $f_0(p) = (\vec p_1,\dotsc, \vec p_n, \vec p_n ;q)$. We put $(z_{n+1}, \lambda_{n+1}) = (z_n, \lambda_n)$.

We further assume $q$ is close to $p_n$ so that we can take $w = z_n$. Then, by the expansion \eqref{eq:bidiff} of the basic bi-differential $\psi$, there exists some holomorphic $1$-form $g = g(z_1,\dotsc, z_{n+1}, w)$ so that we have
\begin{align}
	&2\pi\sqrt{-1} \cdot f_\varepsilon^*(\Psi^{n,n+1} + \Psi^{n+1,n}) \notag\\
	&= f_\varepsilon^*\Big( \frac{dz_n}{z_n - z_{n+1}} - \frac{dz_n}{z_n -w} + \frac{dz_{n+1}}{z_{n+1} - z_n} - \frac{dz_{n+1}}{z_{n+1} -w} + g\Big)\notag\\
	&= f_\varepsilon^*\Big( \frac{dz_n - dz_{n+1}}{z_n - z_{n+1}} - \frac{dz_n}{z_n -w} - \frac{dz_{n+1}}{z_{n+1} -w} + g\Big) \notag\\
	&= \frac{dz_n - d(z_n + \varepsilon\lambda_n + O(\varepsilon^2))}{z_n - (z_n + \varepsilon\lambda_n + O(\varepsilon^2))} - \frac{dz_n}{z_n -w} - \frac{d(z_n + \varepsilon\lambda_n + O(\varepsilon^2))}{(z_n + \varepsilon\lambda_n + O(\varepsilon^2)) - w} + f^*_\varepsilon(g) \notag\\
	&= \frac{- \varepsilon d\lambda_n + O(\varepsilon^2)}{- \varepsilon\lambda_n + O(\varepsilon^2)} - \frac{2dz_n}{z_n -w} + f^*_\varepsilon(g) + O(\varepsilon) \notag\\
	&= \frac{d\lambda_n}{\lambda_n} - \frac{2dz_n}{z_n -w} + f^*_\varepsilon(g) + O(\varepsilon).\label{eq:dlambda}
\end{align}
Since $p\notin D$, the functions $\lambda_n$ and $z_n - w$ never vanish, so the right-hand side uniformly converges on any chosen compact set containing $p$ as $\varepsilon \to 0$ by Lemma \ref{lem:holconv}. If $q$ is not close to $p_n$, the terms where $w$ appears in the above computation are absorbed into $g$, so we have the same conclusion. This shows (1).

\noindent (2) Apart from the exceptional two terms, we know that the limit is exactly the pull-back by $f_0$. Since $f_0$ is a holomorphic map, and the singular loci of $\omega_{a_1\dotsc a_s b}^n$, $\omega_{a_1\dotsc a_s b}^{n+1}$, $\Psi^{n,j}_{a_1\dotsc a_s}$ and $\Psi^{n+1,j}_{a_1\dotsc a_s}$ are transverse to the image of $f_0$, the pull-backs are also meromorphic. The poles are inherited from them, so they lie along $\sum_{1\leq i< n}\tilde \Delta^{(n+1)}_{i,n} + \sum_{1\leq i\leq n}\tilde \Delta^{(n+1)}_{i,n+1}$.

On the other hand, the contribution of
\begin{align*}
	f_\varepsilon^*(\Psi^{n,n+1}) t_{n, n+1} +  f_\varepsilon^*(\Psi^{{n+1},n}) t_{{n+1},n}
\end{align*}
also converges to a meromorphic $1$-form by the above computation on local coordinates. We have the pole $\frac{d\lambda_n}{\lambda_n}$, which is a simple pole along $0_{T\tilde C}^{(n)}$. This exhausts all possible singularities.

\noindent (3) For the residue at $\{p\in T\tilde C^n \times \tilde C: p_i = p_n\}$, we have
\begin{align*}
	\mathrm{res}_{p_i=p_n} \eta^{(n)} &= \mathrm{res}_{p_i=p_n} \beta^{(n+1)}_n + \mathrm{res}_{p_i=p_{n+1}} \beta^{(n+1)}_{n+1}\\
	&= \frac{t_{in} + t_{i,n+1}}{2\pi\sqrt{-1}}
\end{align*}
by observing the corresponding poles of $\beta_i^{(n+1)}$. On the other hand, the residue at $0_{T\tilde C}^{(n)}$ is already computed above, which is $\frac{t_{n,n+1}}{2\pi\sqrt{-1}}$.

\noindent (4) By the $(F_g)^n$-equivariance of $f_\varepsilon$, we have
\begin{align*}
	\gamma_a^{(i)}(\eta^{(n)}) &= \lim_{\varepsilon \to 0} f_\varepsilon^*\big(\gamma_a^{(i)}(\beta_n^{(n+1)} + \beta_{n+1}^{(n+1)})\big)\\
	&= \lim_{\varepsilon \to 0} f_\varepsilon^*(e^{\ad x_i^a}(\beta_n^{(n+1)} + \beta_{n+1}^{(n+1)}))\\
	&= e^{\ad x_i^a}\eta^{(n)}
\end{align*}
for $1\leq i<n$, and
\begin{align*}
	\gamma_a^{(n)}(\eta^{(n)}) &= \lim_{\varepsilon \to 0} f_\varepsilon^*\big(\gamma_a^{(n)}\gamma_a^{(n+1)}(\beta_n^{(n+1)} + \beta_{n+1}^{(n+1)})\big)\\
	&= \lim_{\varepsilon \to 0} f_\varepsilon^*(e^{\ad (x_n^a + x_{n+1}^a)}(\beta_n^{(n+1)} + \beta_{n+1}^{(n+1)}))\\
	&= e^{\ad (x_n^a + x_{n+1}^a)}\eta^{(n)}.
\end{align*}

\noindent (5) By Lemma \ref{lem:beta}, we have
\begin{align*}
	d^{(n+1)} \eta^{(n)} &= \lim_{\varepsilon \to 0} f_\varepsilon^*\big( d^{(n+2)} (\beta_n^{(n+1)} + \beta_{n+1}^{(n+1)}) \big)\\
	&=\psi^n\cdot (1-g)(t_{nn} + t_{n+1, n+1}).
\end{align*}
This completes the proof.\qed\\

From now on, we compute an explicit form of $\eta^{(n)}$.

\begin{lemma}\label{lem:int}
Let $\vec{\mathcal{A}}_a$ be the loop in $T\tilde C$ defined in \eqref{eq:vecA}. Then, we have
\begin{align*}
	\int_{\vec p \in \vec{\mathcal{A}}_a} \eta^{(n)}(\vec p_1, \dotsc, \vec p_{n-1}, \vec p; q) &= \frac{\ad x_n^a}{e^{\ad x_n^a} - 1}(y_n^a) + \frac{\ad x_{n+1}^a}{e^{\ad x_{n+1}^a} - 1}(y_{n+1}^a) - t_{n, n+1},
\end{align*}
where $\vec p_i = (p_i, v_i) \in T\tilde C$, and all $p_i$ are in a neighbourhood of $q \in \mathrm{Int}(\mathcal{D})$.
\end{lemma}
\noindent Proof. By definition, we have
\begin{align*}
	\int_{\vec p \in \vec{\mathcal{A}}_a} \eta^{(n)}(\vec p_1, \dotsc, \vec p_{n-1}, \vec p; q) &= \lim_{\varepsilon \to 0} \int_{\vec p \in \vec{\mathcal{A}}_a} f_\varepsilon^*(\beta_n^{(n+1)} + \beta_{n+1}^{(n+1)})(\vec p_1, \dotsc, \vec p_{n-1}, \vec p; q)\\
	&= \lim_{\varepsilon \to 0} \int_{(p,v) \in \vec{\mathcal{A}}_a} (\beta_n^{(n+1)} + \beta_{n+1}^{(n+1)})(p_1, \dotsc, p_{n-1}, p, \mathrm{Exp}_p(\varepsilon v); q)\\
	&=: I.
\end{align*}
Since we have $d^{(n+1)}\beta_n^{(n+1)} = d^{(n)}\beta_{n+1}^{(n+1)}$ (see the proof of Theorem \ref{thm:kzb}), the integral of $\beta_n^{(n+1)} + \beta_{n+1}^{(n+1)}$ only depends on the homotopy class of the path within $\{(p_1,\dotsc, p_{n-1})\}\times \tilde C^2\times \{q\} \subset \tilde C^{n+2}$.

Now take a point $\vec{\mathbf{p}} = (\mathbf{p}, \mathbf{v})$ on $\vec{\mathcal{A}}_a$. Since we have already established the convergence, we only consider $\varepsilon$ with $\arg(\varepsilon)$ close to $\pi/4$. Since $\mathcal{A}_a$ is, by definition, a smooth boundary of $\mathcal{D}$ with the induced orientation, so we have $\mathrm{Exp}_\mathbf{p}(\varepsilon \mathbf{v}) \in \mathrm{Int}(\mathcal{D})$. Then, the homology class of the loop
\[
	\vec{\mathcal{A}}_a \to \tilde C^{n+1}: (p,v)\mapsto (p_1, \dotsc, p_{n-1}, p, \mathrm{Exp}_p(\varepsilon v); q)
\]
is represented by the sum of the homology classes of the two loops
\begin{align*}
	\ell_1&\colon \vec{\mathcal{A}}_a \to \tilde C^{n+1}: (p,v)\mapsto (p_1, \dotsc, p_{n-1}, p, \mathrm{Exp}_\mathbf{p}(\varepsilon \mathbf{v}); q) \mbox{ and}\\
	\ell_2 &\colon \vec{\mathcal{A}}_a \to \tilde C^{n+1}: (p,v) \mapsto (p_1, \dotsc, p_{n-1}, \mathbf{p}, \mathrm{Exp}_p(\varepsilon v); q).
\end{align*}
Shifting $\ell_2$ by an isotopy in the direction of $-\varepsilon v$ at each $p\in \mathcal{A}_a$, where $v$ is the velocity vector of $\mathcal{A}_a$ at $p$, it can be replaced with
\begin{align*}
	\ell_3 &\colon \vec{\mathcal{A}}_a \to \tilde C^{n+1}: (p,v) \mapsto (p_1, \dotsc, p_{n-1}, \mathrm{Exp}_\mathbf{p}(-\varepsilon \mathbf{v}), p; q).
\end{align*}
Since $\beta_i^{(n+1)}$ is a meromorphic section of $K^{(i)}_{\tilde C}$, we have
\begin{align*}
	I &= \lim_{\varepsilon \to 0} \int_{\ell_1 + \ell_3} (\beta_n^{(n+1)} + \beta_{n+1}^{(n+1)})\\
	&= \lim_{\varepsilon \to 0} \Big(\int_{\ell_1} \beta_n^{(n+1)} + \int_{\ell_3} \beta_{n+1}^{(n+1)}\Big)
\end{align*}
as the only moving point of $\ell_1$ (resp.\ $\ell_3$) is in the $n$-th (resp.\ $(n+1)$-st) component.

The first integral is equal to $\frac{\ad x_n^a}{e^{\ad x_n^a} - 1}(y_n^a)$ by Lemma \ref{lem:beta} (3) since all the non-moving points
\[
	p_1,\dotsc, p_{n-1}, \mathrm{Exp}_\mathbf{p}(\varepsilon \mathbf{v})
\]
are in $\mathrm{Int}(\mathcal{D})$. On the other hand, in the second integral, we have $\mathrm{Exp}_p(-\varepsilon \mathbf{v})\notin\mathrm{Int}(\mathcal{D})$ but $\gamma_a^{-1}(\mathrm{Exp}_p(-\varepsilon \mathbf{v})) \in\mathrm{Int}(\mathcal{D})$. Therefore,
\begin{align*}
	\int_{\ell_3} \beta_{n+1}^{(n+1)} &= \int_{(\gamma_a^{(n)})^{-1}(\ell_3)} (\gamma_a^{(n)})^{-1}(\beta_{n+1}^{(n+1)})\\
	&= \int_{(\gamma_a^{(n)})^{-1}(\ell_3)} e^{-\ad x_n^a} (\beta_{n+1}^{(n+1)})\\
	&= e^{-\ad x_n^a} \frac{\ad x_{n+1}^a}{e^{\ad x_{n+1}^a} - 1}(y_{n+1}^a)\\
	&= \frac{\ad x_{n+1}^a}{e^{\ad x_{n+1}^a} - 1}(y_{n+1}^a) + \frac{e^{-\ad x_n^a} - 1}{\ad x_n^a} \cdot \ad x_n^a \cdot \frac{\ad x_{n+1}^a}{e^{\ad x_{n+1}^a} - 1}(y_{n+1}^a)
\end{align*}
Using $[x_n^a, y_{n+1}^a] = t_{n,n+1}$, $[x_n^a, x_{n+1}^a] = 0$ and $[x_n^a + x_{n+1}^a, t_{n,n+1}] = 0$, we have
\begin{align*}
	\textrm{(RHS)} &= \frac{\ad x_{n+1}^a}{e^{\ad x_{n+1}^a} - 1}(y_{n+1}^a) + \frac{e^{-\ad x_n^a} - 1}{\ad x_n^a} \frac{\ad x_{n+1}^a}{e^{\ad x_{n+1}^a} - 1}(t_{n,n+1})\\
	&= \frac{\ad x_{n+1}^a}{e^{\ad x_{n+1}^a} - 1}(y_{n+1}^a) + \frac{e^{-\ad x_n^a} - 1}{\ad x_n^a} \frac{-\ad x_n^a}{e^{-\ad x_n^a} - 1}(t_{n,n+1})\\
	&= \frac{\ad x_{n+1}^a}{e^{\ad x_{n+1}^a} - 1}(y_{n+1}^a) - t_{n,n+1}.
\end{align*}
Combining these two yields the result.\qed\\[-7pt]

\begin{lemma}\label{lem:eta}
We have $\eta^{(n)} = (\beta_n^{(n)})^{1,\dotsc, n-1, n(n+1)} + \varphi^n t_{n,n+1}$.
\end{lemma}
\noindent Proof. Let $P = (\beta_n^{(n)})^{1,\dotsc, n-1, n(n+1)} + \varphi^n t_{n,n+1} - \eta^{(n)} $.
\begin{itemize}
	\item First of all, we show $d^{(n+1)}(P) = 0$: by Lemmata \ref{lem:beta}, \ref{lem:phi} and \ref{lem:degen}, we have
\begin{align*}
	d^{(n+1)}(P) &= (\psi^n (1-g)t_{nn})^{1,\dotsc, n-1, n(n+1)} + \psi^n (2g-2)t_{n,n+1} - \psi^n (1-g)(t_{nn} + t_{n+1, n+1})\\
	&= 0.
\end{align*}
	\item Next, we show $\gamma_a^{(i)}(P) = (e^{\ad x_i^a})^{1,\dotsc, n-1, n(n+1)} P$. By Lemmata \ref{lem:beta} (1) and \ref{lem:degen} (4), $(\beta_n^{(n)})^{1,\dotsc, n-1, n(n+1)}$ and $\eta^{(n)}$ satisfies this transformation rule. As for $\varphi^n t_{n,n+1}$, we use the fact that $t_{n,n+1}$ commutes with $(x_i^a)^{1,\dotsc, n-1, n(n+1)}$ for any $1\leq i\leq n$ to conclude that it also satisfies the same rule.
	\item Now we show that $P(\vec p_1, \dotsc, \vec p_n; q)$ is holomorphic near the diagonals $\{p_i = p_n\}$. Firstly, any pole of $(\beta_n^{(n)})^{1,\dotsc, n-1, n(n+1)}$ along $\{p_i = p_n\}$ for some $1\leq i<n$ has the residue $\Big(\frac{t_{in}}{2\pi\sqrt{-1}}\Big)^{1,\dotsc, n-1, n(n+1)}$. All of these are precisely cancelled with the poles of $\eta^{(n)}$ by Lemma \ref{lem:degen} (3). We also know that the residue of $P$ along $\{p_i = q\}$ is zero since we showed that $d^{(n+1)}(P) = 0$ and $q$ is arbitrary.
	\item Finally, we show $\int_{\vec p\in\vec{\mathcal{A}_a}} P(\vec p_1, \dotsc, \vec p_{n-1}, \vec p; q) = 0$ for $p_i, q\in\mathrm{Int}(\mathcal{D})$. This follows from
\[
	\Big(\frac{\ad x_n^a}{e^{\ad x_n^a} - 1}(y_n^a)\Big)^{1,\dotsc, n-1, n(n+1)} = \frac{\ad x_n^a}{e^{\ad x_n^a} - 1}(y_n^a) + \frac{\ad x_{n+1}^a}{e^{\ad x_{n+1}^a} - 1}(y_{n+1}^a) - t_{n, n+1},
\]
which, in turn, follows from the relations $[x_n^a, y_{n+1}^a] = t_{n,n+1}$, $[x_n^a, x_{n+1}^a] = 0$ and $[x_n^a + x_{n+1}^a, t_{n,n+1}] = 0$.
\end{itemize}

In summary, we have
\begin{gather*}
	d^{(n+1)}(P) = 0, \quad \gamma_a^{(i)}(P) = (e^{\ad x_i^a})^{1,\dotsc, n-1, n(n+1)} P \quad\mbox{and}\quad \int_{\vec p\in\vec{\mathcal{A}_a}} P(\vec p_1, \dotsc, \vec p_{n-1}, \vec p; q)  = 0.
\end{gather*}
By Lemma \ref{lem:phi} (1) and the computation in \eqref{eq:dlambda}, we know that the $d\lambda$-component in $\varphi$ and $\eta^{(n)}$ cancel each other out, so $P$ desecnds to $\tilde C^n$. Now we set up the non-negative grading on $\mathfrak{t}_{g,n}^f$ by declaring $\deg(x_i^a) = \deg(t_{ij}) = 1$ and $\deg(y_i^a) = 0$. Suppose $P \neq 0$ and take the smallest $d\geq 0$ such that the degree-$d$ part $P_d$ is non-zero. By extracting the degree-$d$ part of $\gamma_a^{(i)}(P) = (e^{\ad x_i^a})^{1,\dotsc, n-1, n(n+1)} P$, we have
\[
	\gamma_a^{(i)}(P_d) = P_d
\]
for all $1\leq a\leq g$ and $1\leq i\leq n$. This implies that $P_d \in H^0(C^n, K_C^{(n)}\otimes \hat{\mathfrak{t}}_{g,n}^f)$, which leads to $P_d = 0$ as all integrals along $A$-loops vanish, yielding the contradiction. This shows $P = 0$ and completes the proof.\qed\\

\noindent \textbf{Proof of Theorem \ref{thm:conv}.} (1) The invariance under $\mathrm{S}_n$-action is clear from the definition of $\vec\alpha_\mathrm{KZ}^{(n)}$.

\noindent (2) We have already seen that $f_\varepsilon^* (\beta_n^{(n+1)} + \beta_{n+1}^{(n+1)})$ converges to $\eta^{(n)}$. As for $f_\varepsilon^*(\beta_i^{(n+1)})$ with $1\leq i< n$, Lemma \ref{lem:holconv} is applicable to all summands, and we have
\begin{align*}
	\lim_{\varepsilon \to 0}f_\varepsilon^*(\beta_i^{(n+1)}) &= \sum_{s\geq 0} \sum_{1\leq a_1,\dotsc a_s, b\leq g} \omega_{a_1\dotsc a_s b}^i \ad x_i^{a_1} \cdots \ad x_i^{a_s}(y_i^b)\\
	&\qquad + \sum_{s\geq 0} \sum_{1\leq a_1,\dotsc a_s \leq g} \Big( \sum_{\substack{1\leq j< n\\j\neq i}}  \Psi^{i,j}_{a_1\dotsc a_s} \ad x_i^{a_1} \cdots \ad x_i^{a_s}(t_{ij})\\
	&\hspace{110pt} + \Psi^{i,n}_{a_1\dotsc a_s} \ad x_i^{a_1} \cdots \ad x_i^{a_s}(t_{in}) + \Psi^{i,n}_{a_1\dotsc a_s} \ad x_i^{a_1} \cdots \ad x_i^{a_s}(t_{i,n+1})\Big)\\
	&= (\beta_i^{(n)})^{1,\dotsc, n-1, n(n+1)}.
\end{align*}
Combining with Lemma \ref{lem:eta}, we have
\begin{align*}
	\lim_{\varepsilon \to 0} f_\varepsilon^*(\vec\alpha_\mathrm{KZ}^{(n+1)}) &= \sum_{1\leq i\leq n+1} \lim_{\varepsilon \to 0}f_\varepsilon^*\Big( \beta_i^{(n)} + \varphi^i \frac{t_{ii}}{2} \Big)\\
	&= \sum_{1\leq i < n} \Big( \beta_i^{(n)} +\varphi^i \frac{t_{ii}}{2} \Big)^{1,\dotsc, n-1, n(n+1)}\\
	&\qquad + (\beta_n^{(n)})^{1,\dotsc, n-1, n(n+1)} + \varphi^n t_{n,n+1} +  \varphi^n \Big(\frac{t_{nn}}{2} + \frac{t_{n+1, n+1}}{2}\Big)\\
	&= \sum_{1\leq i \leq n} \Big( \beta_i^{(n)} +\varphi^i \frac{t_{ii}}{2} \Big)^{1,\dotsc, n-1, n(n+1)}\\
	&= (\vec\alpha_\mathrm{KZ}^{(n)})^{1,\dotsc, n-1, n(n+1)}.
\end{align*}
This completes the proof.\qed\\[-7pt]

\begin{theorem}\label{thm:operadic}
$\{\vec{\alpha}_\mathrm{KZ}^{(n)}\}_{n\geq 0}$ is a unique lift of $\{\alpha_\mathrm{KZ}^{(n)}\}_{n\geq 0}$ satisfying (1) and (2) in Theorem \ref{thm:conv}.
\end{theorem}
\noindent Proof. Any lift of $\alpha_\mathrm{KZ}^{(n)}$ is of the form $\vec{\alpha}_\mathrm{KZ}^{(n)} + A^{(n)}$, where $A^{(n)}$ is some $1$-form on $T\tilde C^n$ whose coefficients are in $\Ker q = \bigoplus_{1\leq i\leq n} \mathbb{C}t_{ii}$. Put $A^{(n)} = \sum_{1\leq i\leq n} A^{(n)}_i t_{ii}$. By the condition (2), the limit $\lim_{\varepsilon \to 0} f_\varepsilon^*(A^{(n+1)})$ must exist, and we necessarily have 
\[
	\lim_{\varepsilon \to 0} f_\varepsilon^*(A^{(n+1)}) = (A^{(n)})^{1,\dotsc, n-1, n(n+1)}.
\]
The coefficient of $t_{n,n+1}$ in the right-hand side is $A^{(n)}_n$, while $t_{n, n+1}$ does not appear in the left-hand side.
This shows $A^{(n)}_n = 0$. By the condition (1), we can swap the indices by $\mathrm{S}_n$-action, and we have $A^{(n)}_i = 0$ for any $1\leq i\leq n$.\qed

\begin{remark}
The above proof shows that an operadic lift is always unique, if exists, for any family of connections on $\mathrm{Conf}_n(C)$.
\end{remark}

\begin{remark} 
The notion of operadicity for a family of connections is due to Gonzalez, introduced in the unpublished manuscript \cite{gonzop}, but the notion itself has, of course, essentially appeared in the original paper of Drinfeld \cite{drinfeld}. 

According to Gonzalez, the full definition of the operadicity requires data of: (1) a variety $X_n$ with a well-behaving divisor $\partial X_n$ such that the family $(X_n, \partial X_n)$ constitutes an operad, and (2) a (flat) connection on $X_n$ extended to $\partial X_n$ as a log-connection. Such a family of connections is operadic if connections are compatible with the pull-back by the operadic composition maps. While the detailed algebro-geometric discussion is out of the scope of this paper as it requires a suitable complex compactification of $\mathrm{Conf}_n^\mathrm{fr}(C)$, namely the Fulton--MacPherson compactification \cite{fm}, Theorem \ref{thm:conv} would be the essential step in proving the operadicity in Gonzalez' sense.

If we reduce $\mathrm{Conf}_n^\mathrm{fr}(C)$ to the framed configuration space with \textit{unit} tangent vectors
\[
	\mathrm{Conf}_n^\mathrm{ufr}(C) := \mathrm{Conf}_n^\mathrm{fr}(C)/(\mathbb{R}_{>0})^n,
\]
we have natural maps
\[
	F\colon \mathrm{Conf}_n^\mathrm{fr}(C) \to \mathrm{Conf}_n^\mathrm{ufr}(C) \to \partial\overline{\mathrm{Conf}_{2n}(C)}
\]
where $\overline{\mathrm{Conf}_{2n}(C)}$ is the Axlelrod--Singer compactification \cite{axcomp}, which is a real manifold with corners, and the second arrow is the inclusion to the boundary strata. We also have the map
\[
	F_\varepsilon \colon \mathrm{Conf}_n^\mathrm{fr}(C) \to \mathrm{Conf}_{2n}^\mathrm{fr}(C) \twoheadrightarrow \mathrm{Conf}_{2n}(C) \hookrightarrow \overline{\mathrm{Conf}_{2n}(C)}
\]
where the first arrow is doubling all points with the parameter $\varepsilon$, similarly to $f_\varepsilon$. Then, we have $F = \lim_{\varepsilon \to 0}F_\varepsilon$, and the theorems above may be restated as that the connection $\alpha_\mathrm{KZ}$ extends to the boundary of $\overline{\mathrm{Conf}_{2n}(C)}$. This is how we obtained the formula in Definition \ref{def:conn} in the first place.

Markl showed in \cite{markl} that the framed version of the Axlelrod--Singer compactification of $d$-dimensional real manifold has a natural action of the operad, in the category of topological spaces, of compactified configuration spaces of framed points in $\mathbb{R}^d$. Indeed, the operad $\mathbf{PaB}^f_\Sigma$ in the category of groupoids in the next section is obtained by taking the fundamental groupoid of Markl's one for $d = 2$.\\
\end{remark}

\section{The Framed Genus $g$ KZB Associator}\label{sec:gonzalez}

In this section, we recall the definition of a genus $g$ Gonzalez--Drinfeld (GD) associator as well as canonical solutions for the KZB system, and show that the framed KZB connection $\vec\alpha_\mathrm{KZ}^{(n)}$, in fact, gives an example of a GD associator. For the original definition of genus $g$ GD associators, see \cite{gonzalez}.

\subsection{Parenthesised Permutations}
We recall some definitions to fix the notation.

\begin{definition}
	A \textit{rooted planar binary tree} is a pair $(\Gamma, r)$ of a finite planar uni-trivalent tree $\Gamma$ and its univalent vertex $r$. We set
	\begin{align*}
		\mathrm{Edge}(\Gamma) &= \{\mbox{all edges of }\Gamma\},\\
		\mathrm{Node}(\Gamma) &= \{\mbox{all trivalent vertices of }\Gamma\},\mbox{ and }\\
		\mathrm{Leaf}(\Gamma, r) &= \{\mbox{all univalent vertices of } \Gamma\} \setminus \{r\}.
	\end{align*}
	For each $t \in \mathrm{Node}(\Gamma)$, we refer to edges incident to $t$ as $\mathsf{up}(t), \mathsf{left}(t), \mathsf{right}(t)$ in this cyclic order, where $\mathsf{up}(t)$ is the one closest to $r$. The set of all vertices is equipped with the partial order $\leq$, where $x\leq y$ means that $x$ is in the downstream of $y$; $r$ is hence the largest element. We denote by $x\lor y$ the join (or the supremum) of $x$ and $y$, which always exists. We futher set 
	\begin{align*}
		\mathrm{Leaf}(t) &= \{i\in \mathrm{Leaf}(\Gamma, r): i\leq t\},\\
		\mathrm{Leaf}(e) &= \{i\in \mathrm{Leaf}(\Gamma, r): i\mbox{ is contained in the rooted subtree below } e\},\mbox{ and}\\
		\mathrm{REdge}(\Gamma, r) &= \{\mathsf{right}(t): t \in \mathrm{Node}(\Gamma)\} \cup \{e_r\} \subset \mathrm{Edge}(\Gamma),
	\end{align*}
	where $e\in\mathrm{Edge}(\Gamma)$, $t\in\mathrm{Node}(\Gamma)$, and $e_r$ is the edge incident to $r$. 
	Furthermore, we define the map
	\[
		L\colon \mathrm{Edge}(\Gamma) \to \mathrm{Leaf}(\Gamma, r)
	\]
	by declaring that $L(e)$ is the unique leaf reached via the downward path in $\Gamma$ starting from $e$ by choosing $\mathsf{left}(t)$ at each $t\in\mathrm{Node}(\Gamma)$ along the way. The restriction
	\[
		L|_{\mathrm{REdge}(\Gamma, r)}\colon \mathrm{REdge}(\Gamma, r)\to \mathrm{Leaf}(\Gamma, r)
	\]
	is bijective.
\end{definition}
\begin{definition}
Let $I$ be a finite set. A \textit{parenthesised permutation} over $I$ is a triple $s = (\Gamma, r, \ell)$ where $(\Gamma, r)$ is a rooted planar binary tree, and $\ell\colon \mathrm{Leaf}(\Gamma, r) \to I$ is a bijection called a \textit{labelling}. We denote by $\mathrm{PP}(I)$ the set of parenthesised permutations over $I$. We put
\[
	\hat{\ell} = \ell\circ L\colon \mathrm{Edge}(\Gamma) \to I
\]
so that $\hat{\ell}|_{\mathrm{REdge}(\Gamma, r)}$ is again bijective.
\end{definition}

We write a parenthesised permutation as $(3(21))4$, which is understood as the rooted planar binary tree in the following picture.
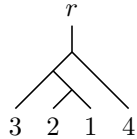
\[\begin{tikzpicture}[scale=0.25, line width = 0.6]
	\node[below] at (0,0) {3};
	\node[below] at (2,0) {2};
	\node[below] at (4,0) {1};
	\node[below] at (6,0) {4};
	\node[above] at (3,4.4) {$r$};
	\draw (0,0) -- (3, 3) -- (3, 4.4);
	\draw (3,3) -- (6,0);
	\draw (2,2) -- (4,0);
	\draw (2,0) -- (3,1);
\end{tikzpicture}\]

Now let $\Sigma$ be any smooth, connected, oriented surface equipped with a base point $*\in\Sigma$ with a non-zero tangent vector $v$ at $*$. Take a (smooth) local coordinate $z\colon \Sigma \supset U \to \mathbb{C}$ with
\begin{align}\label{eq:loccoord}
	z(*) = 0\quad\mbox{and}\quad (dz)_*(v) = \frac{\partial}{\partial z}\Big|_0 \in T_0\mathbb{C}\,.
\end{align}
In the case of $\Sigma = \mathbb{C}$, we conventionally take $* = 0$, $v = \frac{\partial}{\partial z}\Big|_0$, $U = \mathbb{C}$ and $z = \id_\mathbb{C}\colon \mathbb{C} \to \mathbb{C}$.

\begin{definition}
Let $(z_i)_{i\in I}$ be the local coordinate on $U^I \subset \Sigma^I$ consisting of copies of $z$ above. For $s\in \mathrm{PP}(I)$, we define a new coordinate $(u^s_i)_{i\in I}$ on $U^I \cap \mathrm{Conf}_I(\Sigma)$ by the following rule. For each $i\in I$,
\begin{enumerate}[(1)]
	\item If $(\hat{\ell}|_{\mathrm{REdge}(\Gamma, r)})^{-1}(i) = \mathsf{right}(t_a)$ for a (unique) $t_a \in \mathrm{Node}(\Gamma)$, and
	\begin{enumerate}[(a)]
		\item If there is $t_b \in \mathrm{Node}(\Gamma)$ one edge closer from $t_a$ to $r$, we set
		\[
			u^s_i = \frac{z_{\hat{\ell}(\mathsf{right}(t_a))} - z_{\hat{\ell}(\mathsf{left}(t_a))}}{z_{\hat{\ell}(\mathsf{right}(t_b))} - z_{\hat{\ell}(\mathsf{left}(t_b))}};
		\]
		\item If not, we set $u^s_i = z_{\hat{\ell}(\mathsf{right}(t_a))} - z_{\hat{\ell}(\mathsf{left}(t_a))}$;
	\end{enumerate}
	\item If $(\hat{\ell}|_{\mathrm{REdge}(\Gamma, r)})^{-1}(i) = e_r$, we set $u^s_i = z_{\hat{\ell}(e_r)}$.
\end{enumerate}
\end{definition}

\begin{lemma}\label{lem:zu}
For $s\in\mathrm{PP}(I)$ and $i\in I$, we have
\[
	z_i = u^s_{\hat\ell(e_r)} + \sum_{\substack{t\colon t\in\mathrm{Node}(\Gamma)\\ \ell^{-1}(i)\in\mathrm{Leaf}(\mathsf{right}(t))}}\prod_{\substack{\tau \in\mathrm{Node}(\Gamma)\\ \tau \geq t}} u^s_{\hat\ell(\mathsf{right}(\tau))}.
\]
In particular, each $z_i$ is a polynomial in $(u_j^s)_{j\in I}$ of degree $\geq 1$.
\end{lemma}
\noindent Proof. Let $d$ be the combinatorial distance between $(\hat{\ell}|_{\mathrm{REdge}(\Gamma, r)})^{-1}(i)$ and $r$. We proceed by induction on $d$.
\begin{itemize}
	\item If $d = 0$, it means that $(\hat{\ell}|_{\mathrm{REdge}(\Gamma, r)})^{-1}(i) = e_r$. In this case, we have $u^s_i = z_{\hat{\ell}(e_r)} = z_i$. On the other hand, there are no $t\in\mathrm{Node}(\Gamma)$ with $\ell^{-1}(i)\in\mathrm{Leaf}(\mathsf{right}(t))$. The claim is true in this case.
	\item Suppose $d > 0$. We have $(\hat{\ell}|_{\mathrm{REdge}(\Gamma, r)})^{-1}(i) = \mathsf{right}(t_a)$ for some $t_a \in \mathrm{Node}(\Gamma)$. 
	
	Consider the case (b) in the above. Then, we have
	\[
		z_i = u^s_{\hat{\ell}(\mathsf{right}(t_a))} + z_{\hat{\ell}(\mathsf{left}(t_a))}.
	\]
	Since $(\hat{\ell}|_{\mathrm{REdge}(\Gamma, r)})^{-1}\hat{\ell}(\mathsf{left}(t_a)) = e_r$, we have
	\[
		z_i = u^s_{\hat{\ell}(\mathsf{right}(t_a))} + u^s_{\hat\ell(e_r)},
	\]
	which coincides with the right-hand side of the statement. Next, consider the case (a) in the above, and take $t_b$ accordingly. We have
	\begin{align}\label{eq:ziu}
		z_i = u^s_{\hat{\ell}(\mathsf{right}(t_a))} \cdot (z_{\hat{\ell}(\mathsf{right}(t_b))} - z_{\hat{\ell}(\mathsf{left}(t_b))}) + z_{\hat{\ell}(\mathsf{left}(t_a))}
	\end{align}
	by the definition. The edge $(\hat{\ell}|_{\mathrm{REdge}(\Gamma, r)})^{-1}\hat{\ell}(\mathsf{left}(t_a))$ has combiantorial distance $<d$ since it is placed ``upstream'' of $t_a$. Now we invoke the induction hypothesis to obtain
	\begin{align*}
		z_{\hat{\ell}(\mathsf{left}(t_a))} &= u^s_{\hat\ell(e_r)} + \sum_{\substack{t\colon t\in\mathrm{Node}(\Gamma)\\ L(\mathsf{left}(t_a))\in\mathrm{Leaf}(\mathsf{right}(t))}}\prod_{\substack{\tau \in\mathrm{Node}(\Gamma)\\ \tau \geq t}} u^s_{\hat\ell(\mathsf{right}(\tau))}\\
		&= u^s_{\hat\ell(e_r)} - \prod_{\substack{\tau \in\mathrm{Node}(\Gamma)\\ \tau \geq t_a}} u^s_{\hat\ell(\mathsf{right}(\tau))} + \sum_{\substack{t\colon t\in\mathrm{Node}(\Gamma)\\ \ell^{-1}(i)\in\mathrm{Leaf}(\mathsf{right}(t))}}\prod_{\substack{\tau \in\mathrm{Node}(\Gamma)\\ \tau \geq t}} u^s_{\hat\ell(\mathsf{right}(\tau))}\\
		&= u^s_{\hat\ell(e_r)} - u^s_{\hat\ell(\mathsf{right}(t_a))} \prod_{\substack{\tau \in\mathrm{Node}(\Gamma)\\ \tau \geq t_b}} u^s_{\hat\ell(\mathsf{right}(\tau))} + \sum_{\substack{t\colon t\in\mathrm{Node}(\Gamma)\\ \ell^{-1}(i)\in\mathrm{Leaf}(\mathsf{right}(t))}}\prod_{\substack{\tau \in\mathrm{Node}(\Gamma)\\ \tau \geq t}} u^s_{\hat\ell(\mathsf{right}(\tau))}.
	\end{align*}
	On the other hand, the product $\prod_{\substack{\tau \in\mathrm{Node}(\Gamma)\\ \tau \geq t_b}}u^s_{\hat\ell(\mathsf{right}(\tau))} $ is exactly $z_{\hat{\ell}(\mathsf{right}(t_b))} - z_{\hat{\ell}(\mathsf{left}(t_b))}$ by the procedure (a) and (b) above. Substituting this to \eqref{eq:ziu} yields the result.
	\end{itemize}
This completes the proof.\qed

\begin{example}
Take $s = (3(21))4$. We have
\[
	u^s_1 = \frac{z_1 - z_2}{z_2 - z_3},\; u^s_2 = \frac{z_2 - z_3}{z_4 - z_3},\; u^s_3 = z_3,\; u^s_4 = z_4- z_3.
\]
This implies
\[
	z_1 = u^s_4u^s_2 u^s_1 + u^s_4u^s_2 + u^s_3 ,\; z_2 = u^s_4u^s_2 + u^s_3,\; z_3 = u^s_3, \;z_4 = u^s_4 + u^s_3.
\]
\end{example}

\begin{lemma}\label{lem:divis}
Let $s\in\mathrm{PP}(I)$ and $i\neq j \in I$. Then, we have
\[
	z_i - z_j = p_{ij}^s \cdot \prod_{\substack{\tau\in\mathrm{Node}(\Gamma)\\ \ell^{-1}(i)\lor \ell^{-1}(j)\leq \tau}} u^s_{\hat\ell(\mathsf{right}(\tau))}
\]
where $p_{ij}^s$ is a polynomial in $(u^s_i)_{i\in I}$ with $p(0,\dotsc, 0) = \pm 1$. The sign is determined by the relative position of $\ell^{-1}(i)$ to $\ell^{-1}(j)$ within $s$.
\end{lemma}
\noindent Proof. Without loss of generality, we assume $\ell^{-1}(i)$ is on the left side relative to $\ell^{-1}(j)$. By the lemma above, we have
\begin{align*}
	z_i - z_j &= \sum_{\substack{t\colon t\in\mathrm{Node}(\Gamma)\\ \ell^{-1}(i)\in\mathrm{Leaf}(\mathsf{right}(t))}}\prod_{\substack{\tau \in\mathrm{Node}(\Gamma)\\ \tau \geq t}} u^s_{\hat\ell(\mathsf{right}(\tau))} -  \sum_{\substack{t\colon t\in\mathrm{Node}(\Gamma)\\ \ell^{-1}(j)\in\mathrm{Leaf}(\mathsf{right}(t))}}\prod_{\substack{\tau \in\mathrm{Node}(\Gamma)\\ \tau \geq t}} u^s_{\hat\ell(\mathsf{right}(\tau))}\\
	&= \sum_{\substack{t< \ell^{-1}(i)\lor \ell^{-1}(j)\\ \ell^{-1}(i)\in\mathrm{Leaf}(\mathsf{right}(t))}}\prod_{\substack{\tau \in\mathrm{Node}(\Gamma)\\ \tau \geq t}} u^s_{\hat\ell(\mathsf{right}(\tau))} -\prod_{\substack{\tau \in\mathrm{Node}(\Gamma)\\ \tau \geq \ell^{-1}(i)\lor \ell^{-1}(j)}} u^s_{\hat\ell(\mathsf{right}(\tau))} -  \sum_{\substack{t< \ell^{-1}(i)\lor \ell^{-1}(j)\\ \ell^{-1}(j)\in\mathrm{Leaf}(\mathsf{right}(t))}}\prod_{\substack{\tau \in\mathrm{Node}(\Gamma)\\ \tau \geq t}} u^s_{\hat\ell(\mathsf{right}(\tau))}\\
	&= p_{ij}^s\cdot \prod_{\substack{\tau \in\mathrm{Node}(\Gamma)\\ \ell^{-1}(i)\lor \ell^{-1}(j)\leq \tau }} u^s_{\hat\ell(\mathsf{right}(\tau))}
\end{align*}
where
\[
	p_{ij}^s = \sum_{\substack{t< \ell^{-1}(i)\lor \ell^{-1}(j)\\ \ell^{-1}(i)\in\mathrm{Leaf}(\mathsf{right}(t))}}\prod_{\substack{\ell^{-1}(i)\lor \ell^{-1}(j) > \tau \geq t}} u^s_{\hat\ell(\mathsf{right}(\tau))} - 1 - \sum_{\substack{t< \ell^{-1}(i)\lor \ell^{-1}(j)\\ \ell^{-1}(j)\in\mathrm{Leaf}(\mathsf{right}(t))}}\prod_{\ell^{-1}(i)\lor \ell^{-1}(j) > \tau \geq t}u^s_{\hat\ell(\mathsf{right}(\tau))}.
\]
Now we compute $p_{ij}^s(0,\dotsc, 0)$. If there is $t$ such that $t< \ell^{-1}(i)\lor \ell^{-1}(j)$ in the above summations, the above products would not be empty since we can take $\tau$ as $t$, so $p_{ij}^s(0,\dotsc, 0) = -1$. If not, the summation is empty, so we have $p_{ij}^s(0,\dotsc, 0) = -1$. This completes the proof.\qed

\begin{definition}
Let $(u^s_i;\lambda_i)_{i\in I}$ be the coordinate of
\[
	TU^I \cap \mathrm{Conf}^\mathrm{fr}_I(\Sigma)\cong (U^I \cap \mathrm{Conf}_I(\Sigma))\times (\mathbb{C}\setminus\{0\})^I
\]
where $(u^s_i)_{i\in I}$ is taken as above, and $\lambda_i$ is as in \eqref{eq:assoccoord}. For $s\in \mathrm{PP}(I)$, we define $V_s \subset TU^I \cap \mathrm{Conf}^\mathrm{fr}_I(\Sigma)$ as the subset satisfying the conditions
\[
	\begin{cases}
	u_i^s, \lambda_i \in \mathbb{R},\\
	0 < u^s_i < \delta_s & i\in I\mbox{ with } \hat{\ell}^{-1}(i) \neq e_r,\\
	\lambda_i > 0 & i\in I,
	\end{cases}
\]
where we fix $\delta_s > 0$ such that $p_{ij}^s$ in the lemma above does not vanish on $V_s$ for all $i\neq j\in I$. Then, the map
\[
	(u^s_i; \lambda_i)_{i\in I} \colon V_s \to (0,\delta_s)^I \times \mathbb{R}_+^I
\]
is a homeomorphism and therefore $V_s$ is contractible. We take $U_s$ to be a simply-connected small neighbourhood of $V_s$ inside $\mathrm{Conf}^\mathrm{fr}_I(\Sigma)$. Then, it is justified to refer $U_s$ as a ``base point" of the fundamental groupoid of $\mathrm{Conf}^\mathrm{fr}_I(\Sigma)$.
\end{definition}

\subsection{Operad Modules}
Now we explain the operad module $\mathbf{PaB}^f_\Sigma$ of framed braids on $\Sigma$. Recall that an operad (module) is a collection of objects indexed by finite sets.

\begin{definition}
For a finite set $I$, the groupoid $\mathbf{PaB}^f_\Sigma(I)$ is defined as the fundamental groupoid of $\mathrm{Conf}_I^\mathrm{fr}(\Sigma)$ with base points $\{U_s\}_{s\in \mathrm{PP}(I)}$ as above. We denote its Malcev completion over $\mathbb{C}$ by $\widehat{\mathbf{PaB}}{}^f_\Sigma(I)$, and set
\[
	\mathbf{PaB}_\Sigma^f := \{\mathbf{PaB}_\Sigma^f(I)\}_I,
\]
which is a collection over all finite sets.
\end{definition}

For any choice of a local coordinate $z$ as above, the fundamental groupoid $\mathbf{PaB}^f_\Sigma(I)$ is canonically isomorphic to each other. The automorphism group $\mathbf{PaB}^f_\Sigma(I)(s,s)$ at an object $s \in\mathrm{PP}(I)$ is isomorphic to the \textit{framed pure braid group} $\mathrm{PB}_{g,n}^f$ of genus $g$ with $n = \# I$ strands.

In the special case of $\Sigma = \mathbb{C}$, $\mathbf{PaB}_\mathbb{C}^f$ is endowed with a structure of an operad by the doubling (and more general compositions) of braids, and each $\mathbf{PaB}_\Sigma^f$ is an operad module over $\mathbf{PaB}_\mathbb{C}^f$. By the functoriality of the completion, $\widehat{\mathbf{PaB}}{}^f_\Sigma$ is an operad module over $\widehat{\mathbf{PaB}}{}^f_\mathbb{C}$. When $\Sigma$ is closed of genus $g$, $\mathbf{PaB}_\Sigma^f$ is denoted by $\mathbf{PaB}_g^f$ in \cite{gonzalez}, where the choice of genus $g$ surface $\Sigma$ with a framed base point is implicitly fixed.\\

The following is the graded counterpart of $\mathbf{PaB}_\mathbb{C}^f$.

\begin{definition}\label{def:frameddk}
Let $I$ be a finite set.
\begin{itemize}
	\item The \textit{framed Drinfeld--Kohno Lie algebra} $\mathfrak{t}^f_I$ is generated by $t_{ij}$ for $i,j\in I$ together with relations
	\begin{align*}
	\begin{aligned}
		&t_{ij} = t_{ji},&\\
		&[t_{ij}, t_{kl}] = 0& &\mbox{if } \{i,j\}\cap\{k,l\}=\varnothing,\mbox{ and}\\
		&[t_{ij},t_{ik} + t_{jk}] = 0& &\mbox{if } \{i,j\}\cap\{k\}=\varnothing.
	\end{aligned}
	\end{align*}
	We have the natural morphism $\iota\colon \mathfrak{t}^f_I \to \mathfrak{t}_{g,I}^f$ of Lie algebras.
	\item The groupoid $\mathbb{G}\mathbf{PaCD}^f(I)$ is defined as follows: the set of objects is the same as $\mathbf{PaB}^f_\mathbb{C}$. For parenthesised permutations $s$ and $s'$ over $I$, the morphism set $\mathbb{G}\mathbf{PaCD}^f(I)(s,s')$ is $\exp(\hat{\mathfrak{t}}_I^f)$.
\end{itemize}
\end{definition}
Then, $\mathbb{G}\mathbf{PaCD}^f$ is endowed with a structure of an operad similar to $\mathbf{PaB}_\mathbb{C}^f$. For $\Sigma$ closed of genus $g$, the following is the graded counterpart of $\mathbf{PaB}_\Sigma^f$.

\begin{definition}
For a finite set $I$, the groupoid $\mathbb{G}\mathbf{PaCD}^f_g(I)$ is defined as follows: the set of objects is the same as $\mathbf{PaB}^f_\Sigma$. For parenthesised permutations $s$ and $s'$ over $I$, the morphism set $\mathbb{G}\mathbf{PaCD}^f_g(I)(s,s')$ is $\exp(\hat{\mathfrak{t}}_{g,I}^f)$. In this case, $\mathbb{G}\mathbf{PaCD}^f_g$ is an operad module over $\mathbb{G}\mathbf{PaCD}^f$ via $\iota \colon \mathfrak{t}^f_I \to \mathfrak{t}_{g,I}^f$.
\end{definition}

With these objects prepared, Gonzalez \cite{gonzalez} generalised Drinfeld associators to higher genera as follows.

\begin{definition}
Let the underlying surface of $\Sigma$ be closed of genus $g$. A \textit{Gonzalez--Drinfeld associator} for $\Sigma$ is an isomorphism
\[
	W\colon (\widehat{\mathbf{PaB}}_\mathbb{C}{}^f, \widehat{\mathbf{PaB}}{}^f_\Sigma) \xrightarrow{\cong} (\mathbb{G}\mathbf{PaCD}^f, \mathbb{G}\mathbf{PaCD}^f_g)
\]
of (pointed) operad modules over the category of completed groupoids. By definition, $W$ consists of the isomorphism
\[
	\Phi\colon \widehat{\mathbf{PaB}}_\mathbb{C}{}^f \xrightarrow{\cong} \mathbb{G}\mathbf{PaCD}^f
\]
of operads, which we call the \textit{accompanying Drinfeld associator} of $W$, and the isomorphism
\[
	\hat Z\colon \widehat{\mathbf{PaB}}{}^f_\Sigma\xrightarrow{\cong}  \mathbb{G}\mathbf{PaCD}^f_g
\]
of operad modules compatible with $\Phi$.
\end{definition}

\subsection{Main Result}
From now on, we will show the following result, answering affirmatively to Conjecture 3.21 in \cite{gonzalez}. Recall that, in the construction of $\vec\alpha_\mathrm{KZ}^{(n)}$ (or the original $\alpha_\mathrm{KZ}^{(n)}$), we fixed the following data:
\begin{itemize}
	\item a closed Riemann surface $C$ and a base point $*\in C$,
	\item a marking of $(C, *)$, which is an isomorphism $\pi_1(C,*) \cong \pi_g$, and
	\item a base point $\tilde *\in \tilde C$ lifting $*\in C$.
\end{itemize}
We further choose a non-zero tangent vector $v$ at $*\in C$. Put $\Sigma = (C, *, v)$, where $C$ has the induced smooth structure.

\begin{theorem}\label{thm:kzbassoc} 
Given the above data, the framed KZB connection $\{\vec{\alpha}_\mathrm{KZ}^{(n)}\}_{n\geq 0}$ gives rise to a Gonzalez--Drinfeld associator for $\Sigma$. Furthermore, its accompanying Drinfeld associator is the KZ associator.
\end{theorem}

The proof will be given along the $g=1$ case described in Section 4 of \cite{cee}, which is based on the original argument by Drinfeld \cite{drinfeld}. Although the argument is classical, we give a proof for the sake of clarity. We start with the following lemma, which is a variant of Proposition A.3 in \cite{cee}, and gives an improved estimate for the behaviour of the solution.

\begin{lemma}\label{lem:cee}
Let $A$ be a $\mathbb{Z}_{\geq 0}$-graded complete $\mathbb{C}$-algebra with each graded piece finite-dimensional, and denote by $A_{(>0)}$ the positive-degree part. Let $I$ be a finite set, and $u_i$ $(i\in I)$ formal variables. Suppose that $C_i\in A_{(>0)}\,\hat\otimes\, \mathbb{C}[[u_i]]_{i\in I}$ is given for $i\in I$, and
\begin{enumerate}[(1)]
	\item $u_i\frac{\partial C_j}{\partial u_i} - u_j\frac{\partial C_i}{\partial u_j} = [C_i, C_j]$,
	\item $C_i = c_i + u_i C'_i$ for some $c_i \in A_{(>0)}$ and $C'_i\in A_{(>0)}\,\hat\otimes \,\mathbb{C}[[u_i]]_{i\in I}$,
	\item $[C_i|_{u_j = 0}, c_j] = 0$, and
	\item $C_i$ has a positive convergence radius
\end{enumerate}
for any $i, j\in I$. Then, there uniquely exists $F \in A\,\hat{\otimes}\, \mathbb{C}[[u_i]][\log u_i]_{i\in I}$ such that
\begin{enumerate}[(a)]
	\item $u_i\frac{\partial F}{\partial u_i} = C_iF$,
	\item putting $P = \prod_{i \in I} u_i^{c_i}$, $FP^{-1} \in 1 + A_{(>0)}\,\hat{\otimes}\, \mathfrak{m}$, and 
	\item $FP^{-1}$ has a positive convergence radius,
\end{enumerate}
where $\mathfrak{m} = \Ker(\mathbb{C}[[u_i]]_{i\in I} \to \mathbb{C})$ is the augmentation ideal, and $u_i^{c_i} := e^{c_i \log u_i}$. We call such $P$ the principal part of the solution $F$.
\end{lemma}
\noindent Proof. Substituting $u_i = u_j = 0$ to (1) and using (2), we have $[c_i, c_j] = 0$. Therefore, the order of the product in $P$ does not matter. Then, the equations $Y := FP^{-1}$ has to satisfy are
\begin{align*}
	u_i \frac{\partial Y}{\partial u_i} &= u_i \frac{\partial F}{\partial u_i} P^{-1} - F u_i P^{-1}\frac{\partial P}{\partial u_i} P^{-1}\\
	&= C_iF P^{-1} - F P^{-1} c_i \\
	&= C_i Y - Y c_i\\
	&= (C_i - c_i) Y + [c_i, Y]\\
	&= u_i C'_i Y + [c_i, Y]
\end{align*}
for each $i\in I$. Write $Y = \sum_{k\geq 0} Y_{(k)}$ where $Y_{(k)}$ has degree $k$ in $A$. Now the equation above reads
\[
	\frac{\partial Y_{(k)}}{\partial u_i} = \sum_{0\leq d< k} \Big( {C'_i}_{(k-d)} Y_{(d)} + \frac{[{c_i}_{(k-d)}, Y_{(d)}]}{u_i} \Big)
\]
for $i\in I$ and $k\geq 0$. We will show that such $Y_{(k)}\in 1 + A_{(>0)}\,\hat{\otimes}\, \mathfrak{m}$ uniquely exists and $[c_i, Y_{(k)}]|_{u_i = 0} = 0$ holds for all $i\in I$ by induction on $k$.
\begin{itemize}
	\item For $k = 0$, the equation reads $\frac{\partial Y_{(0)}}{\partial u_i} = 0$. We have $Y_{(0)} = 1$. It clearly satisfies the induction hypothesis.
	\item Suppose $k>0$. By the hypothesis $[c_i, Y_{(d)}]|_{u_i = 0} = 0$ for $d < k$, the term $\frac{[{c_i}_{(k-d)}, Y_{(d)}]}{u_i}$ does not contain negative powers of $u_i$. Since the system is integrable by the condition (1), there uniquely exists the solution $Y_{(k)} \in 1 + A_{(>0)}\,\hat{\otimes}\, \mathfrak{m}$ to the above equation normalised by $Y_{(k)}(0,\dotsc, 0) = 0$. We check $[c_i, Y_{(k)}]|_{u_i = 0} = 0$. Fix $i\in I$, and take a total order on $I$ such that $i$ is the smallest. We can identify $I$ with $\{1,\dotsc, n\}$ via this total order, and $i$ corresponds to $1$. Then, we have
	\begin{align*}
		[c_1, Y_{(k)}]|_{u_1 = 0} &= \sum_{0\leq d < k} \Big[c_1, \sum_{1\leq l\leq n} \int_0^{u_l} \Big( {C'_l}_{(k-d)}Y_{(d)} + \frac{[{c_l}_{(k-d)} , Y_{(d)}]}{u_l}\Big)\Big|_{u_{l+1} = \cdots = u_n = 0} \,du_l \Big] \Big|_{u_1 = 0}\\
		&= \sum_{0\leq d < k}  \sum_{2\leq l\leq n} \Big[c_1, \int_0^{u_l} \Big( {C'_l}_{(k-d)}Y_{(d)} + \frac{[{c_l}_{(k-d)} , Y_{(d)}]}{u_l}\Big)\Big|_{u_1 = u_{l+1} = \cdots = u_n = 0} \,du_l \Big].
	\end{align*}
	On the other hand, for $l\geq 2$, we have $[c_1, C_l']|_{u_1 = 0} = 0$ by the conditions (2) and (3), and $[c_1, Y_{(d)}]|_{u_1 = 0} = 0$ by the induction hypothesis. Combining with $[c_1, c_l] = 0$, the right-hand side is zero. This completes the induction.
\end{itemize}
So far, we have proved (a) and (b). Since $Y$ is constructed by iterated integrals of various $C_i$, the convergene radius is at least the smallest of $C_i$'s. This proves (c) and completes the proof. \qed

\begin{lemma}\label{lem:cstcomm}
For $s = (\Gamma, r, \ell) \in \mathrm{PP}(I)$ and $t\in \mathrm{Node}(\Gamma)$, we define  $c^s_t\in \mathfrak{t}_{g, I}^f$ by
\begin{align*}
	c^s_t = \frac12 \sum_{\substack{i,j\in \mathrm{Leaf}(t)\\ i\neq j}} t_{\ell(i), \ell(j)}
\end{align*}
Then, we have $[c^s_t, c^s_{t'}] = 0$ for $t, t'\in  \mathrm{Node}(\Gamma)$.
\end{lemma}
\noindent Proof. Suppose neither $t\leq t'$ nor $t'\leq t$ holds. In this case, $\mathrm{Leaf}(t)$ and $\mathrm{Leaf}(t')$ are disjoint, so we have $[c^s_t, c^s_{t'}] = 0$ by the second relation in Definition \ref{def:frameddk}. Now suppose $t\leq t'$, in which case we have $\mathrm{Leaf}(t) \subset \mathrm{Leaf}(t')$ and 
\begin{align*}
	c^s_{t'} &= c_t^s + \frac12 \sum_{\substack{k,l\in \mathrm{Leaf}(t')\setminus  \mathrm{Leaf}(t)\\ i\neq j}} t_{\ell(k), \ell(l)} + \sum_{p\in \mathrm{Leaf}(t)} \sum_{k\in \mathrm{Leaf}(t') \setminus  \mathrm{Leaf}(t)} t_{\ell(p), \ell(k)}.
\end{align*}
Then, $c_t^s$ commutes with the second term by the same reason as above. For any $i\neq j\in \mathrm{Leaf}(t)$ and $k\in \mathrm{Leaf}(t') \setminus  \mathrm{Leaf}(t)$, we can see that $t_{\ell(i), \ell(j)}$ commutes with $ \sum_{p\in \mathrm{Leaf}(t)} t_{\ell(p), \ell(k)}$ using the last two relations in Definition \ref{def:frameddk}.\qed\\

Now we take $z$ to be a holomorphic local coordinate $z\colon C \supset U \to \mathbb{C}$ satisfying \eqref{eq:loccoord}. Let $(z_1,\lambda_1; \dotsc; z_n,\lambda_n)$ be the local coordinate on $\mathrm{Conf}_n^\mathrm{fr}(C)$ consisting of the copies of $(z,\lambda)$ associated with $z$ as in \eqref{eq:assoccoord}. Set
\begin{align*}
	\vec{\omega}_\mathrm{KZ}^{(n)} &= \sum_{1\leq i<j\leq n} \frac{t_{ij}}{2\pi\sqrt{-1}} d\log(z_i - z_j) + \sum_{1\leq i\leq n} \frac{t_{ii}}{4\pi\sqrt{-1}}d\log \lambda_i,\mbox{ and}\\
	 \zeta^{(n)} &= \vec{\alpha}_\mathrm{KZ}^{(n)} - \vec{\omega}_\mathrm{KZ}^{(n)}
\end{align*}
on $U^n$. The $1$-form $\zeta^{(n)}$ is holomorphic on $U^n$ due to the definition of $\vec\alpha_\mathrm{KZ}^{(n)}$ and the expansion \eqref{eq:bidiff}, while $\vec{\omega}_\mathrm{KZ}^{(n)}$ is the \textit{framed KZ connection} appeared in Section 2.2.4 of \cite{gonzalez}.

\begin{lemma}\label{lem:kzucoord}
Let $I = \{1,\dotsc, n\}$. For $s\in\mathrm{PP}(I)$, we have
\[
	 \vec{\alpha}_\mathrm{KZ}^{(n)} = \sum_{t\in\mathrm{Node}(\Gamma)} \frac{c_t^s}{2\pi\sqrt{-1}} d\log u_{\hat{\ell}(\mathsf{right}(t))}^s + \sum_{1\leq i\leq n} \frac{t_{ii}}{4\pi\sqrt{-1}}d\log \lambda_i + (\textrm{holomorphic in }(u_i^s))
\]
on $U_s$.
\end{lemma}
\noindent Proof. Using Lemma \ref{lem:divis}, the first term of $\vec{\omega}_\mathrm{KZ}^{(n)}$ becomes
\begin{align*}
	&\sum_{1\leq i<j\leq n} \frac{t_{ij}}{2\pi\sqrt{-1}} d\log\Big(p_{ij}^s \cdot \prod_{\substack{t\in\mathrm{Node}(\Gamma)\\ \ell^{-1}(i)\lor \ell^{-1}(j)\leq t}} u^s_{\hat\ell(\mathsf{right}(t))}\Big)\\
	&= \sum_{1\leq i<j\leq n} \frac{t_{ij}}{2\pi\sqrt{-1}} \Big( d\log p_{ij}^s +  \sum_{\substack{t\in\mathrm{Node}(\Gamma)\\ \ell^{-1}(i)\lor \ell^{-1}(j)\leq t}} d\log u^s_{\hat\ell(\mathsf{right}(t))}\Big).
\end{align*}
Since we have $p_{ij}^s(0,\dotsc, 0) \neq 0$ by Lemma \ref{lem:divis}, $d\log p_{ij}^s$ is holomorphic at $(u^s_i) = (0,\dotsc, 0)$. For $t\in\mathrm{Node}(\Gamma)$, the condition $\ell^{-1}(i)\lor \ell^{-1}(j)\leq t$ is equivalent to $\ell^{-1}(i), \ell^{-1}(j)\in\mathrm{Leaf}(t)$, so we have
\begin{align*}
	\sum_{1\leq i<j\leq n} \frac{t_{ij}}{2\pi\sqrt{-1}} \sum_{\substack{t\in\mathrm{Node}(\Gamma)\\ \ell^{-1}(i)\lor \ell^{-1}(j)\leq t}} d\log u^s_{\hat\ell(\mathsf{right}(t))} &= \sum_{t\in\mathrm{Node}(\Gamma)} d\log u^s_{\hat\ell(\mathsf{right}(t))}  \sum_{\substack{1\leq i<j\leq n\\\ell^{-1}(i)\lor \ell^{-1}(j)\leq t}} \frac{t_{ij}}{2\pi\sqrt{-1}}\\
	&= \sum_{t\in\mathrm{Node}(\Gamma)} d\log u^s_{\hat\ell(\mathsf{right}(t))} \frac{c^s_t}{2\pi\sqrt{-1}}.
\end{align*}
By Lemma \ref{lem:zu}, $\zeta^{(n)}$ is also holomorphic in $(u_i^s)$-coordinate. This completes the proof. \qed

\begin{definition}
For $s\in\mathrm{PP}(I)$, we define $P_s \colon U_s \to \exp(\hat{\mathfrak{t}}_{g,I}^f)$ by
\[
	P_s = \prod_{t\in \mathrm{Node}(\Gamma)} (u_{\hat{\ell}(\mathsf{right}(t))}^s)^{\frac{c^s_t}{2\pi\sqrt{-1}}} \cdot \prod_{i\in I} \lambda_i^{\frac{t_{ii}}{4\pi\sqrt{-1}}}
\]
The order of the products does not matter since $c_t^s$ commutes with each other by Lemma \ref{lem:cstcomm} and $t_{ii}$ is central.
\end{definition}

\begin{example}
Take $s = (3(21))4$. We have
\begin{align*}
	P_s &= (u_4^s)^\frac{t_{12} + t_{13} + t_{14} + t_{23} + t_{24} + t_{34}}{2\pi\sqrt{-1}}(u_2^s)^\frac{t_{12} + t_{13} + t_{23}}{2\pi\sqrt{-1}} (u_1^s)^\frac{t_{12}}{2\pi\sqrt{-1}}\lambda_1^{\frac{t_{11}}{4\pi\sqrt{-1}}}\lambda_2^{\frac{t_{22}}{4\pi\sqrt{-1}}}\lambda_3^{\frac{t_{33}}{4\pi\sqrt{-1}}}\lambda_4^{\frac{t_{44}}{4\pi\sqrt{-1}}}\\
	&= (z_4- z_3)^\frac{t_{14} + t_{24} + t_{34}}{2\pi\sqrt{-1}} (z_2 - z_3)^\frac{t_{13} + t_{23}}{2\pi\sqrt{-1}} ( z_1 - z_2)^\frac{t_{12}}{2\pi\sqrt{-1}} \lambda_1^{\frac{t_{11}}{4\pi\sqrt{-1}}}\lambda_2^{\frac{t_{22}}{4\pi\sqrt{-1}}}\lambda_3^{\frac{t_{33}}{4\pi\sqrt{-1}}}\lambda_4^{\frac{t_{44}}{4\pi\sqrt{-1}}}.
\end{align*}
\end{example}

\begin{lemma}\label{lem:ratioanal}
Let $I = \{1,\dotsc, n\}$ and $s\in \mathrm{PP}(I)$.
\begin{enumerate}[(a)]
	\item There uniquely exists a solution
\[
	F_s\colon U_s \to \exp(\hat{\mathfrak{t}}_{g,n}^f)
\]
to $dF_s = \vec\alpha_\mathrm{KZ}^{(n)} F_s$ such that $F_s P_s^{-1}$ is holomorphically extended to a neighbourhood of $(u_i^s)_{i\in I} = (0,\dotsc, 0)$ with $(F_s P_s^{-1})(0,\dotsc, 0) = 1$.
\item There uniquely exists a solution
\[
	\hat F_s\colon U_s \to \exp(\hat{\mathfrak{t}}_{g,n+1}^f)
\]
to $d\hat F_s =(\vec\alpha_\mathrm{KZ}^{(n)})^{1,\dotsc, n-1, n(n+1)} \hat F_s$ such that $\hat F_s (P_s^{-1})^{1,\dotsc, n-1, n(n+1)}$ is holomorphically extended to a neighbourhood of $(u_i^s)_{i\in I} = (0,\dotsc, 0)$ with $(\hat F_s (P_s^{-1})^{1,\dotsc, n-1, n(n+1)})(0,\dotsc, 0) = 1$, which is $F_s^{1,\dotsc, n-1, n(n+1)}$.
\end{enumerate}
\end{lemma}
\noindent Proof. (a) We want to apply Lemma \ref{lem:cee}, so we take $A$ to be the completed universal enveloping algebra $\hat U(\mathfrak{t}_{g,n}^f)$, and define $C_i$ by
\[
	 \vec\alpha_\mathrm{KZ}^{(n)} = \sum_{1\leq i\leq n} \Big( C_i d\log u_i^s  + \frac{t_{ii}}{4\pi\sqrt{-1}} d\log \lambda_i \Big)
\] 
on $U_s$. We check the conditions (1)--(4). (1) follows from the flatness of $\vec\alpha_\mathrm{KZ}^{(n)}$. (2) follows from Lemma \ref{lem:kzucoord}, and we have
\[
	c_i = \begin{cases}
		\frac{c_t^s}{2\pi\sqrt{-1}} & \mbox{if } \hat{\ell}(\mathsf{right}(t)) = i\mbox{ for some } t\in\mathrm{Node}(\Gamma),\\
		0 & \mbox{otherwise}.
	\end{cases}
\]
Let us show (3): $[C_i|_{u_j = 0}, c_j] = 0$. Fix $j\in I$. If there are no $t\in\mathrm{Node}(\Gamma)$ with $ \hat{\ell}(\mathsf{right}(t)) = j$, we have $c_j = 0$ and the claim is trivial. Suppose $ \hat{\ell}(\mathsf{right}(t_j)) = j$ for a (unique) $t_j \in\mathrm{Node}(\Gamma)$. By Lemma \ref{lem:divis}, the substitution $u_j = 0$ translates, in terms of $(z_i)$-coordinates, to that all $z_k$ for $k\in\mathrm{Leaf}(t_j)$ are set to be equal. Let $I/\ell(\mathrm{Leaf}(t_j))$ be the quotient set obtained by collapsing $\ell(\mathrm{Leaf}(t_j))$ to a single point, and
\[
	m\colon \mathfrak{t}_{g,I/\ell(\mathrm{Leaf}(t_j))}^f \to \mathfrak{t}_{g,I}^f
\]
 be the operadic composition map associated to the quotient map $I \to I/\mathrm{Leaf}(t_j)$. By repeatedly applying Theorem \ref{thm:conv} to $\vec\alpha_\mathrm{KZ}^{(n)}$ to degenerate $z_k$'s to a single point, the coefficients in $C_i|_{u_j = 0}$ is contained in the image of $m$, so $C_i|_{u_j = 0}$ commutes with $c_j = c^s_{t_j}$. For (4), we can take $\delta$ as the convergence radius. Thus, we can take $F$ and $P$ satisfying (a)--(c) of Lemma \ref{lem:cee}. We set $F_s = F \prod_{i\in I} \lambda_i^{\frac{t_{ii}}{4\pi\sqrt{-1}}}$. Since we also have $P_s = P\prod_{i\in I} \lambda_i^{\frac{t_{ii}}{4\pi\sqrt{-1}}}$ in this case, $F_s$ is the solution to the framed KZB system with the properties we want.

Similarly, we can show (b) by applying $(-)^{1,\dotsc, n-1, n(n+1)}$ everywhere in the above proof. \qed

\begin{lemma}\label{lem:diaghol}
Let $I = \{1,\dotsc, n\}$ and $s\in\mathrm{PP}(I)$. We put $\tilde s = s^{1,\dotsc, n-1, n(n+1)}$ and define another function $G_{\tilde s}$ on $U_{\tilde s}$ by\vspace{-10pt}
\begin{align}\label{eq:Gs}
	G_{\tilde s} &= F_{\tilde s}\cdot J ,\quad \mbox{where } J = (z_{n+1} - z_{n})^{\frac{-t_{n,n+1}}{2\pi\sqrt{-1}}} \lambda_n^{\frac{t_{n,n+1}}{4\pi\sqrt{-1}}}\lambda_{n+1}^{\frac{t_{n,n+1}}{4\pi\sqrt{-1}}}.
\end{align}
Then, $G_{\tilde s}$ is holomorphic near the diagonal $\{z_n = z_{n+1}\}$ in $TU^{n+1}$ and satisfies $f_0^*(G_{\tilde s}) = (F_s)^{1,\dotsc, n-1, n(n+1)} $.
\end{lemma}
\noindent Proof. Since the labels $n$ and $n+1$ are placed in an innermost parenthesis in $\tilde s$, the diagonal $\{z_n = z_{n+1}\}$ is equal to the locus $\{u_{n+1} = 0\}$. Since $c^s_t = t_{n,n+1}$ with $\hat{\ell}(\mathsf{right}(t)) = n+1$, the singular part of $P_{\tilde s}$ along the diagonal $\{z_n = z_{n+1}\}$ is $(z_{n + 1} - z_n)^\frac{t_{n,n+1}}{2\pi\sqrt{-1}}$. The factor $J$ precisely contains the inverse of the singular part of $P_{\tilde s}$ along the diagonal $\{z_n = z_{n+1}\}$. By Lemma \ref{lem:ratioanal} (a), $G_{\tilde s}$ is holomorphic near $\{z_n = z_{n+1}\}$, and $f_0^*(G_{\tilde s})$ is well-defined on $U_s$. Then, applying Lemma \ref{lem:holconv} to $\omega = dG_{\tilde s}$ yields $\lim_{\varepsilon \to 0} f_\varepsilon^*(dG_{\tilde s}) = f_0^*(dG_{\tilde s})$.

By taking the exterior derivative of \eqref{eq:Gs}, we obtain the equation that $G_{\tilde s}$ satisfies:
\begin{align*}
	dG_{\tilde s} &= dF_{\tilde s}\cdot J + F_{\tilde s}\cdot dJ\\
	&= \vec\alpha_\mathrm{KZ}^{(n + 1)} F_{\tilde s}\cdot J + F_{\tilde s}\cdot J \cdot \frac{t_{n,n+1}}{4\pi\sqrt{-1}}\Big(\frac{-2 d(z_{n+1} - z_n)}{z_{n+1} - z_n} + \frac{d\lambda_n}{\lambda_n} + \frac{d\lambda_{n+1}}{\lambda_{n+1}}\Big)\\
	&= \vec\alpha_\mathrm{KZ}^{(n + 1)} G_{\tilde s} + G_{\tilde s}\cdot  \frac{t_{n,n+1}}{4\pi\sqrt{-1}}\Big(\frac{-2 d(z_{n+1} - z_n)}{z_{n+1} - z_n} + \frac{d\lambda_n}{\lambda_n} + \frac{d\lambda_{n+1}}{\lambda_{n+1}}\Big).
\end{align*}
Therefore,
\begin{align*}
	d(f_0^*(G_{\tilde s})) &= f_0^*(dG_{\tilde s})\\
	&= \lim_{\varepsilon \to 0} f_\varepsilon^*(dG_{\tilde s})\\
	&= \lim_{\varepsilon \to 0} f_\varepsilon^* \left( \vec\alpha_\mathrm{KZ}^{(n + 1)} G_{\tilde s} + G_{\tilde s}\cdot  \frac{t_{n,n+1}}{4\pi\sqrt{-1}}\Big(\frac{-2 d(z_{n+1} - z_n)}{z_{n+1} - z_n} + \frac{d\lambda_n}{\lambda_n} + \frac{d\lambda_{n+1}}{\lambda_{n+1}}\Big)\right)\\
	&= (\vec\alpha_\mathrm{KZ}^{(n)})^{1,\dotsc, n-1, n(n+1)} f_0^*(G_{\tilde s}) + 0.
\end{align*}
As for the last equality, the first term is due to Theorem \ref{thm:conv}, and the second term vanishes by a similar calculation to \eqref{eq:dlambda}. Therefore, $f_0^*(G_{\tilde s})$ and $(F_s)^{1,\dotsc, n-1, n(n+1)}$ satisfies the same equation $dH = (\vec\alpha_\mathrm{KZ}^{(n)})^{1,\dotsc, n-1, n(n+1)} H$ for $H \colon U_s \to \exp(\hat{\mathfrak{t}}_{g,n+1}^f)$. In addition, we have
\begin{align*}
	f_0^*(F_{\tilde s}P_{\tilde s}^{-1}) &= f_0^*(F_{\tilde s} J \cdot J^{-1}P_{\tilde s}^{-1})\\
	&= f_0^*(G_{\tilde s}) (P_s^{-1})^{1,\dotsc, n-1, n(n+1)}.
\end{align*}
Since the left-hand side is already holomorphic, so is the right-hand side. Moreover, evaluating the above at $(u_i^s)_{i\in I} = (0,\dotsc, 0)$, we have
\begin{align*}
	f_0^*(G_{\tilde s}) (P_s^{-1})^{1,\dotsc, n-1, n(n+1)}(0,\dotsc, 0) &= f_0^*(F_{\tilde s}P_{\tilde s}^{-1}(0,\dotsc, 0)) = 1.
\end{align*}
By Lemma \ref{lem:ratioanal} (b), we conclude that $f_0^*(G_{\tilde s}) = (F_s)^{1,\dotsc, n-1, n(n+1)} $.\qed

\begin{lemma}\label{lem:globalR}
Let $I = \{1,\dotsc, n\}$. There is a unique solution $R\colon U^n \to \exp(\hat{\mathfrak{t}}_{g,n}^f)$ to
	\[
		dR = \zeta^{(n)} R + [\vec{\omega}_\mathrm{KZ}^{(n)}, R]
	\]
	such that it is analytic near $(z_i)_{i\in I} = (0,\dotsc, 0)$ and $R(0,\dotsc, 0) = 1$. 
\end{lemma}
\noindent Proof. The strategy is similar to Lemma \ref{lem:cee}. Write $R= \sum_{k\geq 0} R_{(k)}$ with $R_{(k)}$ degree $k$ in $U(\hat{\mathfrak{t}}_{g,n}^f)$. We show that:
\begin{enumerate}[(1)]
	\item there uniquely exists $R_{(k)}$ such that it is analytic, satisfies $dR_{(k)} = \sum_{0\leq d<k} \zeta^{(n)}_{(k-d)} R_{(d)} + [\vec{\omega}_\mathrm{KZ}^{(n)}{}_{(k-d)}, R_{(d)}]$ and $R(0,\dotsc, 0) = 1$; and
	\item $[\vec{\omega}_\mathrm{KZ}^{(n)}, R_{(k)}]$ is regular on $U^n$,
\end{enumerate}
by the induction on $k$. By the definition of $\vec{\omega}_\mathrm{KZ}^{(n)}$, (2) implies $R_{(k)} \in \Ker(\ad t_{ij}) + (z_i - z_j)R'_{(k)}$ with some regular $R'_{(k)}$, for any $i,j\in I$.
\begin{itemize}
	\item For $k = 0$, we have $R_{(0)} = 1$. In this case, (1) and (2) are clear.
	\item Suppose $k>0$. By the induction hypothesis (2) for $d<k$, we can take a unique analytic solution $R_{(k)}$ satisfying (1). We will show (2) for $k$. Take $i\neq j\in I$ and focus on the $i,j$-diagonal $\Delta_{i,j}^{(n)}\cap U^n$. By the $\mathrm{S}_n$-symmetry, we can assume $(i,j) = (1,2)$. We have
	\begin{align*}
		&[t_{12} d\log(z_1- z_2), R_{(k)}] = \sum_{0\leq k < d} d\log(z_1- z_2) \Big[t_{12},\\
		&\qquad \int_{(0,\dotsc, 0)}^{(z_1,z_1,z_3,\dotsc, z_n)} \Big( \zeta^{(n)}_{(k-d)} R_{(d)} + [\vec{\omega}_\mathrm{KZ}^{(n)}{}_{(k-d)}, R_{(d)}] \Big) + \int_{(z_1,z_1,z_3\dotsc, z_n)}^{(z_1,z_2, z_3,\dotsc, z_n)} \Big( \zeta^{(n)}_{(k-d)} R_{(d)} + [\vec{\omega}_\mathrm{KZ}^{(n)}{}_{(k-d)}, R_{(d)}] \Big) \Big].
	\end{align*}
	Denote the first- and the second integral by $I_1$ and $I_2$. For $I_1$, we integrate over a path from $(0,\dotsc, 0)$ to $(z_1,z_1,z_3\dotsc, z_n)$ within $\Delta_{1,2}^{(n)}\cap U^n$. Then, we have 
	\[
		[t_{12}, \zeta^{(n)}_{(k-d)}]|_{\Delta_{1,2}^{(n)}\cap U^n} = 0
	\]
	since $\vec{\alpha}_\mathrm{KZ}^{(n)}$ is operadic so that the restriction $ \zeta^{(n)}|_{\Delta_{1,2}^{(n)}\cap U^n}$ lies in $\Ker(\ad t_{12})$. We also have 
	\[
		[t_{12}, R_{(d)}]|_{\Delta_{1,2}^{(n)}\cap U^n} = [t_{12}, (z_1 - z_2)R'_{(d)}]|_{\Delta_{1,2}^{(n)}\cap U^n} = 0.
	\]
	Furthermore, we have 
	\begin{align*}
		[t_{12}, [\vec{\omega}_\mathrm{KZ}^{(n)}{}_{(k-d)}, R_{(d)}]]|_{\Delta_{1,2}^{(n)}\cap U^n} &= [[t_{12}, \vec{\omega}_\mathrm{KZ}^{(n)}{}_{(k-d)}], R_{(d)}] + [\vec{\omega}_\mathrm{KZ}^{(n)}{}_{(k-d)}, [t_{12},R_{(d)}]]\Big|_{\Delta_{1,2}^{(n)}\cap U^n}\\
		&= \sum_{i\geq 3} [[t_{12}, \big(t_{1i}d\log(z_1 - z_i) + t_{2i}d\log(z_2 - z_i)\big)_{(k-d)}], R_{(d)}]\Big|_{\Delta_{1,2}^{(n)}\cap U^n}\\
		&\qquad + [\vec{\omega}_\mathrm{KZ}^{(n)}{}_{(k-d)}, [t_{12},(z_1 - z_2)R'_{(d)}]]\Big|_{\Delta_{1,2}^{(n)}\cap U^n}\\
		&= \sum_{i\geq 3} [[t_{12}, (t_{1i} + t_{2i})d\log(z_1 - z_i)_{(k-d)}], R_{(d)}] \\
		&\qquad + [(z_1 - z_2)\vec{\omega}_\mathrm{KZ}^{(n)}{}_{(k-d)}, [t_{12},R'_{(d)}]]\Big|_{\Delta_{1,2}^{(n)}\cap U^n}.
	\end{align*}
	In the right-hand side, we have $[t_{12}, t_{1i} + t_{2i}] = 0$ and also $(z_1 - z_2)\vec{\omega}_\mathrm{KZ}^{(n)}|_{\Delta_{1,2}^{(n)}} = 0$. Therefore, we have $d\log(z_1- z_2) [t_{12},I_1]= 0$. On the other hand, $I_2$ is of order $O(z_1 - z_2)$ since it is the integral from $(z_1,z_1,z_3\dotsc, z_n)$ to $(z_1,z_2,z_3\dotsc, z_n)$. This shows $d\log(z_1- z_2) [t_{12}, I_2]$ is regular near $\Delta_{1,2}^{(n)}\cap U^n$.
\end{itemize}
This completes the induction and the proof.\qed\\

\noindent\textbf{Proof of Theorem \ref{thm:kzbassoc}.} Let $I = \{1,\dotsc, n\}$.

\noindent\textbf{Step 1:} We shall construct the morphism $Z\colon \mathbf{PaB}_\Sigma^f(n) \to \mathbb{G}\mathbf{PaCD}_g^f(n)$ of groupoid for $n\geq 0$. We take the solution
\[
	F_s\colon U_s \to \exp(\hat{\mathfrak{t}}_{g,n}^f)
\]
to the framed KZB system $dF_s = \vec\alpha_\mathrm{KZ}^{(n)} F_s$ as above. For $b \in \mathbf{PaB}_\Sigma^f(s, s')$, we define $Z(b) \in \mathbb{G}\mathbf{PaCD}_g^f$ by
\[
	F_{s'} Z(b) = F_s\ast b,
\]
where $F_s\ast b$ is the solution $F_s$ transported along a path given by $b$. Here, the flatness of the framed KZB connection implies that the resulting map $Z\colon \mathbf{PaB}_\Sigma^f(n) \to \mathbb{G}\mathbf{PaCD}_g^f(n)$ is a well-defined groupoid homomorphism for each $n$.

\noindent\textbf{Step 2:} We will show that this defines a map of operad modules. Let $b\in \mathbf{PaB}_\Sigma^f(s, s')$ be a braid with $n$ strands. We show that:
\begin{enumerate}[(1)]
	\item If the image of the braid $b$ on $\Sigma$ is contained in $U^n$, and therefore can be regarded as $b\in \mathbf{PaB}_\mathbb{C}^f(n)$ via the local coordinate $(z_1, \dotsc z_n)\colon U^n \to \mathbb{C}^n$, then $Z(b)$ agrees with the evaluation by the composition
	\[
		\mathbf{PaB}_\mathbb{C}^f(n) \xrightarrow{\Phi_\mathrm{KZ}} \mathbb{G}\mathbf{PaCD}^f(n) \xrightarrow{\iota} \mathbb{G}\mathbf{PaCD}^f_g(n)
	\]
	of the KZ associator $\Phi_\mathrm{KZ}$ and the natural map $\iota\colon \mathfrak{t}_n^f \to \mathfrak{t}_{g,n}^f$;
	\item $Z$ is compatible with the doubling of braids; and
	\item $Z$ is compatible with the deletion of strands.
\end{enumerate}
	For (1), let $H_s\colon U_s \to \exp(\hat{\mathfrak{t}}_n^f)$ be the unique solution to $dH_s = \vec{\omega}_\mathrm{KZ}^{(n)} H_s$ with the principal part $P_s$ for $s\in \mathrm{PP}(I)$, whose existence and uniqueness is proved similarly to the above $F_s$. Then, $R_s := F_sH_s^{-1}$ satisfies the equation
	\begin{align*}
		dR_s &= (dF_s) H_s^{-1}  - F_s H_s^{-1} dH_s H_s^{-1}\\
		&= \vec{\alpha}_\mathrm{KZ}^{(n)} F_sH_s^{-1} - F_s H_s^{-1}\vec{\omega}_\mathrm{KZ}^{(n)}\\
		&= \vec{\alpha}_\mathrm{KZ}^{(n)} R_s - R_s\vec{\omega}_\mathrm{KZ}^{(n)}\\
		&= \zeta^{(n)} R_s + [\vec{\omega}_\mathrm{KZ}^{(n)}, R_s]
	\end{align*}
	on $U_s$, and it is unique amongst such that it is analytic in the $(u_i^s)$-coordinate and $R_s(0,\dotsc, 0) = 1$. Indeed, if $R'_s$ is another such solution, $R'_sH_s$ satisfies the same equation as $F_s$ with the same principal part $P_s$, and hence equal to $F_s$ by Lemma \ref{lem:ratioanal} (a). On the other hand, by Lemma \ref{lem:globalR}, $R|_{U_s}$ is also such a solution since Lemma \ref{lem:zu} says that $u_i^s = 0$ for all $i\in I$ implies $z_i = 0$ for all $i\in I$. Therefore, we have $R|_{U_s} = R_s$ so that $R_s = R_{s'}$ in $U_s \cap U_{s'}$ for any $s, s'\in\mathrm{PP}(I)$, and therefore $R_s * b = R_{s'}$. This implies
	\begin{align*}
		Z(b) &= F_{s'}^{-1}\cdot (F_s * b)\\
		&= H_{s'}^{-1}R_{s'}^{-1}\cdot ((R_sH_s) * b)\\
		&= H_{s'}^{-1}R_{s'}^{-1}\cdot R_{s'}(H_s * b)\\
		&= H_{s'}^{-1}\cdot (H_s * b)\\
		&= \iota(\Phi_\mathrm{KZ}(b)),
	\end{align*}
	which shows (1).

	For (2), we proceed as follows. By the $\mathrm{S}_n$-symmetry, we can freely choose the label for the doubling, which we take it to be $n$. Therefore, our goal is to prove $Z(b^{1,\dotsc, n-1, n(n+1)} ) = Z(b)^{1,\dotsc,n-1,n(n+1)}$. We put
\[
	h = J^{-1} (J *b^{1,\dotsc,n-1,n(n+1)}) 
\]
on $U_{\tilde s}$. Then, we have
\begin{align*}
	Z(b^{1,\dotsc,n-1,  n(n+1)} ) &= F_{\tilde s'}^{-1} \cdot (F_{\tilde s}\ast b^{1,\dotsc,n-1,n(n+1)})\\
	&= J G_{\tilde s'}^{-1} \cdot (G_{\tilde s}\ast b^{1,\dotsc,n-1,n(n+1)})  (J^{-1} *b^{1,\dotsc,n-1,n(n+1)})\\
	&= J G_{\tilde s'}^{-1} \cdot (G_{\tilde s}\ast b^{1,\dotsc,n-1,n(n+1)}) J^{-1} h^{-1}\\
	&=: Q.
\end{align*}
Since $G_{\tilde s}$ is holomorphic near $\{z_n = z_{n+1}\}$ by Lemma \ref{lem:diaghol}, the analytic continuation by the path $b^{1,\dotsc,n-1,n(n+1)}$ in $TC^{n+1}$ can be done by taking the path $f_0(b)$, which is homotopic to $b^{1,\dotsc,n-1,n(n+1)}$:
\[
	G_{\tilde s}\ast b^{1,\dotsc,n-1,n(n+1)} = G_{\tilde s}\ast f_0(b).
\]
Recall that $f_0\colon TC^n\to TC^{n+1}$ in Definition \ref{def:fvarepsilon} is a holomorphic map. We have the identity 
\[
	f_0^*(F)*b = f_0^*(F * f_0(b))
\]
between the analytic continuation of any holomorphic function $F$ on an open set in $TC^{n+1}$ containing $f_0(b)$. Therefore, we have
\begin{align*}
	f_0^*(G_{\tilde s}\ast b^{1,\dotsc,n-1,n(n+1)}) &= f_0^*(G_{\tilde s}\ast f_0(b))\\
	&= f_0^*(G_{\tilde s})\ast b\\
	&= (F_s)^{1,\dotsc, n-1, n(n+1)} \ast b\\
	&= (F_s * b)^{1,\dotsc, n-1, n(n+1)}\\
	&= (F_{s'}Z(b))^{1,\dotsc, n-1, n(n+1)}\\
	&= f_0^*(G_{\tilde s'}) Z(b)^{1,\dotsc, n-1, n(n+1)}.
\end{align*}
by the definition of $Z$ and Lemma \ref{lem:diaghol}. Therefore, we obtain
\[	
	G_{\tilde s}\ast b^{1,\dotsc,n-1,n(n+1)} = G_{\tilde s'} Z(b)^{1,\dotsc, n-1, n(n+1)} + O(z_{n+1} - z_n)
\]
near the diagonal $\{z_n = z_{n+1}\}$ in $U_{\tilde s'}$. Since $J$ and $Z(b)^{1,\dotsc, n-1, n(n+1)}$ commute, we have
\begin{align*}
	Q &= J G_{\tilde s'}^{-1} \cdot (G_{\tilde s'} Z(b)^{1,\dotsc, n-1, n(n+1)} + O(z_{n+1} - z_n)) J^{-1} h^{-1}\\
	&= \Big( Z(b)^{1,\dotsc, n-1, n(n+1)} + J G_{\tilde s'}^{-1} O(z_{n+1} - z_n)J^{-1}  \Big)h^{-1}.
\end{align*}
Since $J$ is an infinite series of $t_{n, n+1}\log(z_{n+1} - z_n)$, each graded piece of $J G_{\tilde s'}^{-1} O(z_{n+1} - z_n)J^{-1} $ is of the form $(z_{n+1} - z_n)\cdot  p(\log(z_{n+1} - z_n) )$, where $p$ is some polynomial. Since $Q$ is a constant, we evaluate $Q$ at the limit $z_{n+1} - z_n \to +0$ to obtain
\[	
	Q = Z(b)^{1,\dotsc, n-1, n(n+1)} h^{-1}|_{z_{n+1} - z_n \to +0}.
\]

We compute the limit of $h$. Consider the universal cover $\mathcal{T}$ of $TC$ with $(z,\lambda)$ extended to a global coordinate. On a fixed lift of $U_{\tilde s}$ to $\mathcal{T}^n$, the function $h = J^{-1} (J *b^{1,\dotsc,n-1,n(n+1)})$ is realised as $J^{-1} \cdot \gamma^{(n)}\gamma^{(n+1)}(J) $ for some $\gamma \in \Aut(\mathcal{T}/C)$. In fact, $\gamma$ is the loop on $C$ trailing the $n$-th point of $b$. Since we have
\[
	\gamma^{(n)}\gamma^{(n+1)} \big( (z_{n+1} - z_n)^{-2}\lambda_n\lambda_{n+1} \big) = (z_{n+1} - z_n)^{-2}\lambda_n\lambda_{n+1} (1 + (z_{n+1} - z_n)\cdot \tilde h)
\]
for some holomorphic function $\tilde h$ on the fixed lift of $U_{\tilde s}$, described in terms of derivatives of $\gamma$, the projection to $TC^n$ yields
\[
	h = (1 + (z_{n+1} - z_n) \cdot \tilde h)^{\frac{t_{n,n+1}}{4\pi\sqrt{-1}}}
\]
on $U_{\tilde s}$. Therefore, we have $h|_{z_{n+1} - z_n \to +0} = 1$. In conclusion, we have 
\[
	Z(b^{1,\dotsc, n-1, n(n+1)} ) = Q = Z(b)^{1,\dotsc,n-1,n(n+1)},
\]
showing the compatibility with the doubling.

For (3), again by $\mathrm{S}_n$-symmetry, we only show $Z(b^{1,\dotsc, n-1, \varnothing} ) = Z(b)^{1,\dotsc,n-1,\varnothing}$, where the Lie algebra map $(-)^{1,\dotsc,n-1,\varnothing}\colon \mathfrak{t}_{g,n}^f \to \mathfrak{t}_{g,n-1}^f$ sends $x_n^a, y_n^a$ and $t_{in}$ to $0$ for $1\leq i\leq n$ and $1\leq a\leq g$. Firstly, let $p_n\colon \tilde C^n \to \tilde C^{n-1}$ the projection forgetting the $n$-th point. We have
\[
	(\vec{\alpha}_\mathrm{KZ}^{(n)})^{1,\dotsc,n-1,\varnothing} = p_n^*(\vec{\alpha}_\mathrm{KZ}^{(n-1)})
\]
from the definition of $\vec{\alpha}_\mathrm{KZ}^{(n)}$. Then, applying $(-)^{1,\dotsc,n-1,\varnothing}$ to $dF_s = \vec\alpha_\mathrm{KZ}^{(n)} F_s$ yields 
\[
	dF_s^{1,\dotsc,n-1,\varnothing} = p_n^*(\vec{\alpha}_\mathrm{KZ}^{(n-1)})
F_s^{1,\dotsc,n-1,\varnothing},
\]
so $p_n^*(F_{\bar s})$ and $F_s^{1,\dotsc,n-1,\varnothing}$ satisfy the same equation. In addition, Lemma \ref{lem:divis} implies that $p_n^*(u^{\bar s}_i)$ $(i\in I\setminus \{n\})$ and $P_s^{1,\dotsc,n-1,\varnothing}p_n^*(P_{\bar s})^{-1}$ are holomorphic near $(u^s_j)_{j\in I} = (0,\dotsc, 0)$, and we have
\[
	p_n^*(u^{\bar s}_i)(0,\dotsc, 0) = 0\quad \mbox{and} \quad P_s^{1,\dotsc,n-1,\varnothing}p_n^*(P_{\bar s})^{-1}(0,\dotsc, 0)  = 1.
\]
This shows
\begin{align*}
	p_n^*(F_{\bar s}) (P_s^{1,\dotsc,n-1,\varnothing})^{-1} &= p_n^*(F_{\bar s} P_{\bar s}^{-1})\cdot p_n^*(P_{\bar s}) (P_s^{1,\dotsc,n-1,\varnothing})^{-1}
\end{align*}
is again holomorphic near $(u^s_j)_{j\in I} = (0,\dotsc, 0)$ and evaluates to $1$ at $(u^s_j)_{j\in I} = (0,\dotsc, 0)$. All of these force $p_n^*(F_{\bar s}) = F_s^{1,\dotsc,n-1,\varnothing}$ on $U_s$ by a similar argument as Lemma \ref{lem:ratioanal}. Then, we have
\begin{align*}
	Z(b)^{1,\dotsc, n-1, \varnothing} &= (F_{s'}^{-1})^{1,\dotsc, n-1, \varnothing} \cdot (F_s^{1,\dotsc, n-1, \varnothing} * b)\\
	&= p_n^*(F_{\bar s}^{-1}) \cdot (p_n^*(F_{\bar s}) * b)\\
	&= p_n^*(F_{\bar s}^{-1}) \cdot (p_n^*(F_{\bar s} * p_n(b)))\\
	&= p_n^*(F_{\bar s}^{-1} \cdot (F_{\bar s} * p_n(b)))\\
	&= Z(b^{1,\dotsc, n-1, \varnothing})
\end{align*}
since $p_n(b)$ is equal to $b^{1,\dotsc, n-1, \varnothing}$ by the definition of the deletion of strands. This shows (3).

\noindent\textbf{Step 3:} We show that $Z$ induces an isomorphism
\[
	\hat Z\colon \widehat{\mathbf{PaB}}{}^f_\Sigma \to  \mathbb{G}\mathbf{PaCD}_g
\]
on the Malcev completion. Since $\widehat{\mathbf{PaB}}{}^f_\Sigma$ is a connected groupoid, we only have to check that it is an isomorphism at some object $s$. We identify the framed pure braid group $\mathrm{PB}^f_{g,n}$ with 
\[
	\mathbf{PaB}_\Sigma^f(s,s) = \pi_1(\mathrm{Conf}^\mathrm{fr}_n(\Sigma), s),
\]
so that the groupoid morphism $Z$ at the object $s$ reads $Z\colon \mathrm{PB}^f_{g,n} \to \exp(\hat{\mathfrak{t}}^f_{g,n})$.  By definition, the filtration on $\mathrm{PB}^f_{g,n}$ is induced by the augmentation ideal of the group algebra $\mathbb{C}[\mathrm{PB}^f_{g,n}]$, and on $\exp(\hat{\mathfrak{t}}^f_{g,n})$ by the lower central series of $\mathfrak{t}^f_{g,n}$. The map $Z$ is automatically compatible with these filtrations, so it induces $\hat Z\colon \widehat{\mathrm{PB}}^f_{g,n} \to \exp(\hat{\mathfrak{t}}^f_{g,n})$ on the completion. We will show that $\hat Z$ is an isomorphism.

Since the case of $g=0$ is the already-known KZ associator, we only deal with the case $g\geq 1$. Let
\[
	\mathrm{gr}\,\hat Z\colon \mathrm{gr}(\widehat{\mathrm{PB}}^f_{g,n}) \to \hat{\mathfrak{t}}^f_{g,n}
\]
be the associated graded map between complete Lie algebras. Then, $\hat Z$ is an isomorphism if and only if $\mathrm{gr}\,\hat Z$ is an isomorphism. By the result of Quillen (p.412 of \cite{quillen}), $\mathrm{gr}(\widehat{\mathrm{PB}}^f_{g,n})$ is isomorphic to $\mathrm{gr}(\mathrm{PB}^f_{g,n}) \otimes_\mathbb{Z} \mathbb{C}$ where $\mathrm{gr}(\mathrm{PB}^f_{g,n}) $ is the associated Lie algebra (over $\mathbb{Z}$) with respect to the lower central series, which is always generated by the degree-$1$ part.

We follow the proof of Theorem 4 in \cite{enriquez}. Firstly, we show that $\mathrm{gr}\,\hat Z$ is surjective. We take $X_a^i, Y_a^i\in \mathrm{PB}^f_{g,n}$ satisfying
\begin{align*}
	\mathrm{PB}^f_{g,n} &\to \pi_1(C^n, s) \cong \pi_g^n: X_a^i \mapsto A_a^{(i)}, Y_a^i \mapsto B_a^{(i)}.
\end{align*}
\begin{itemize}
	\item The case of $g\geq 2$. The Lie algebra $\mathfrak{t}^f_{g,n}$ is generated by $x_i^a$ and $y_i^a$ $(1\leq i\leq n, 1\leq a\leq g)$, and the filtration is the one given by the grading $\deg(x_i^a) = \deg(y_i^a) = 1$.
	We have
	\[
		Z(X_a^i) = \exp(y_i^a + (\textrm{deg}\geq 2))\quad\mbox{and}\quad Z(Y_a^i) = \exp\big(x_i^a + \sum_{1\leq b\leq g} \tau_{ab} y_i^b + (\textrm{deg}\geq 2)\big)
	\]
	by the defining formula of $\vec\alpha_\mathrm{KZ}^{(n)}$, where $\tau_{ab} = \int_{B_a} \omega_b$. Denoting by $[b]_1\in \mathrm{gr}(\widehat{\mathrm{PB}}^f_{g,n})$ the element specified by $b\in \mathrm{PB}^f_{g,n}$ in degree $1$, we obtain $\mathrm{gr}\,\hat Z([X_a^i]_1) = y_i^a$ and $\mathrm{gr}\,\hat Z([Y_a^i]_1) = x_i^a +  \sum_{1\leq b\leq g} \tau_{ab} y_i^b$, and therefore $\mathrm{gr}\,\hat Z$ is surjective.
	\item The case of $g=1$. The Lie algebra $\mathfrak{t}^f_{1,n}$ is generated by $x_i^1$, $y_i^1$ and $t_{ii}$ $(1\leq i\leq n)$, and the filtration is the one given by the grading $\deg(x_i^a) = \deg(y_i^a) = \deg(t_{ii}) = 1$. By taking $X_a^i$ and $Y_a^i$ to be rotation-free with respect to the standard framing on a flat torus, we have $\mathrm{gr}\,\hat Z([X_a^i]_1) = y_i^a$ and $\mathrm{gr}\,\hat Z([Y_a^i]_1) = x_i^a +  \sum_{1\leq b\leq g} \tau_{ab} y_i^b$. Let $F^{i}\in \mathrm{PB}^f_{g,n}$ be the element corresponding to the fibre class of $TC^{(i)}\setminus 0_{TC}^{(i)}$. We have $Z(F^{i}) = e^{t_{ii}/2}$ since the residue of $\vec\alpha_\mathrm{KZ}^{(n)}$ at $\lambda_i=0$ is given by $\frac{t_{ii}}{4\pi\sqrt{-1}}$. Therefore, we have $\mathrm{gr}\,\hat Z([F^i]_1) = t_{ii}/2$, and $\mathrm{gr}\,\hat Z$ is surjective.
\end{itemize}

Next, we show that $\mathrm{gr}\,\hat Z$ is injective. By repeating the calculation in the Lemma 5 of \cite{enriquez} using the presentation of $\mathrm{PB}^f_{g,n}$ given by Theorem 7 in \cite{bg}, we have a Lie algebra morphism
\[
	T\colon \mathfrak{t}_{g,n}^f \to \mathrm{gr}(\mathrm{PB}^f_{g,n})\otimes_\mathbb{Z} \mathbb{C}
\]
satisfying
\[
	T(x_i^a) = [X_a^i]_1,\quad T(y_i^a) = [Y_a^i]_1\quad\mbox{and if } g=1,\; T(t_{ii}) =-2[F^i]_1.
\]
Let $\theta \in\Aut(\mathfrak{t}_{g,n}^f)$ be such that
\[
	\theta(x_i^a) = y_i^a,\quad\theta(y_i^a) =  x_i^a + \sum_{1\leq b\leq g} \tau_{ab} y_i^b\quad\mbox{and}\quad \theta(t_{ij}) = -t_{ij}
\]
for $1\leq i,j\leq n$ and $1\leq a\leq g$. This is well-defined since it respects all relations in the presentation of $\mathfrak{t}_{g,n}^f$. In particular,
\[
	\sum_{1\leq a\leq g} [x_i^a, y_i^a] + \sum_{j\in I\setminus\{i\}} t_{ij} = (g-1) t_{ii}
\]
is respected due to $\tau_{ab} = \tau_{ba}$.

Since degree-$1$ elements $[X_a^i]_1$, $[Y_a^i]_1$ (and $[F^i]_1$ for $g=1$) generates $\mathrm{gr}(\mathrm{PB}^f_{g,n})\otimes_\mathbb{Z} \mathbb{C}$ as a Lie algbra, we have $T\circ\theta^{-1}\circ \mathrm{gr}\,\hat Z = \id$, and therefore $\mathrm{gr}\,\hat Z$ is injective. In conclusion, $\hat Z\colon \widehat{\mathrm{PB}}^f_{g,n} \to \exp(\hat{\mathfrak{t}}^f_{g,n})$ is an isomorphism, and so is $\hat Z\colon \widehat{\mathbf{PaB}}{}^f_\Sigma \to  \mathbb{G}\mathbf{PaCD}_g$.

All of this show that $(\Phi_\mathrm{KZ}, \hat Z)$ is a genus $g$ GD associator.\qed\\

\appendix
\section{Correction to the Proof of Enriquez' Lemma} \label{sec:correction}

In this section, we provide corrected proofs of Lemma 9 (b) in \cite{enriquez}}, whose statement is true but the original proof contains some inaccuracies. We revert back to the notation in \cite{enriquez}, and recall the statement.

\begin{lemma}[Lemma 9 (b) in \cite{enriquez}]\label{lem:lem9}
\[
	(\gamma_a^{(w)} - 1)\omega^{\underline{z} w}_{a_1\dotsc a_s bb} = \sum_{0\leq k\leq s} \frac{(-1)^{k+1}}{(k+1)!} \delta_{a a_s\dotsc a_{s-k+1}} \omega^{\underline{z} w}_{a_1\dotsc a_{s-k}a}.
\]
\end{lemma}

The original proof fails at the last step: the integral $\int_{z\in \mathcal{A}_c} (\mbox{LHS})$ and the sum $\sum_{0\leq k\leq s} \frac{(-1)^{k+1}}{(k+1)!}$ are not zero as claimed.

\subsection{Proof 1}
Our first proof is a direct correction to the original proof, showing they are indeed equal (and non-zero). Let $\tilde D$ be any fundamental domain of the $F_g$-action on $\tilde C$, and $\mathcal{A}_a = \mathcal{A}_a^{\tilde D} $ the $a$-th $A$-loop lying on the boundary of $\tilde D$. Recall that the action of $\gamma\in F_g$ on a function or form is by $\gamma \,f = f\circ \gamma^{-1}$. 

\begin{lemma}\label{lem:integralout}
For $w\in{\tilde D}$, we have
\begin{align*}
	\int_{\mathcal{A}_b^{\tilde D}} \gamma^{(w)}_a \omega^{\underline{z} w}_{a_1\dotsc a_s a a} = \delta_{a_1\dotsc a_sa b} \left(b_{s+2} - \frac{(-1)^s}{s!}\right).
\end{align*}
\end{lemma}
\noindent Proof. Putting $\zeta = \gamma_a^{-1}(z)$, we have
\begin{align*}
	\int_{\mathcal{A}_b^{\tilde D}} \gamma^{(w)}_a \omega^{\underline{z} w}_{a_1\dotsc a_s a a} &= \int_{z\in \mathcal{A}_b^{\tilde D}}  \omega^{\underline{z} \gamma_a^{-1}(w)}_{a_1\dotsc a_s a a}\\
	&= \int_{\zeta\in \gamma_a^{-1}(\mathcal{A}_b^{\tilde D})}  \omega^{\underline{\gamma_a(\zeta)} \gamma_a^{-1}(w)}_{a_1\dotsc a_s a a}\\
	&= \int_{\zeta\in \gamma_a^{-1}(\mathcal{A}_b^{\tilde D})} (\gamma_a^{-1})^{(\zeta)}\omega^{\underline{\zeta} \gamma_a^{-1}(w)}_{a_1\dotsc a_s a a}\\
	&=  \sum_{0\leq k < s+2} \frac{(-1)^k}{k!} \delta_{a a_1\dotsc a_k} \int_{\zeta\in \gamma_a^{-1}(\mathcal{A}_b^{\tilde D})} \omega^{\underline{\zeta} \gamma_a^{-1}(w)}_{a_{k+1}\dotsc a_{s+2}}.
\end{align*}
Here, we put $a_{s+1} = a_{s+2} = a$.

Now put ${\tilde D}' = \gamma_a^{-1}({\tilde D})$. We have $\mathcal{A}_b^{{\tilde D}'} = \gamma_a^{-1}(\mathcal{A}_b^{{\tilde D}})$. We also have $\gamma_a^{-1}(w) \in {\tilde D}'$, so we apply Lemma \ref{lem:integral} to obtain
\begin{align*}
	\int_{\zeta\in \gamma_a^{-1}(\mathcal{A}_b^{\tilde D})} \omega^{\underline{\zeta} \gamma_a^{-1}(w)}_{a_{k+1}\dotsc a_{s+2}} &= \int_{\zeta\in \mathcal{A}_b^{{\tilde D}'}} \omega^{\underline{\zeta} \gamma_a^{-1}(w)}_{a_{k+1}\dotsc a_{s+2}}\\
	 &= \delta_{a_{k+1}\dotsc a_{s+2} b}  b_{s+2-k}.
\end{align*}
This implies, using $a_{s+2} = a$,
\begin{align*}
	\int_{\mathcal{A}_b^{\tilde D}} \gamma^{(w)}_a \omega^{\underline{z} w}_{a_1\dotsc a_s a a} &= \sum_{0\leq k < s+2} \frac{(-1)^k}{k!} \delta_{a a_1\dotsc a_k} \delta_{a_{k+1}\dotsc a_{s+2} b} b_{s+2-k} \\
	&= \delta_{a a_1\dotsc a_{s+2} b}  \sum_{0\leq k < s+2} \frac{(-1)^k}{k!} b_{s+2-k}\\
	&= \delta_{a a_1\dotsc a_{s+2} b}  \sum_{1\leq k \leq s+2} \frac{(-1)^{k-1}}{(k-1)!} b_{s+3-k}.
\end{align*}
Since $\sum_{p\geq 1}\frac{(-1)^{p-1}}{(p-1)!} t^p = t e^{-t}$, the last sum is the coefficient of $t^{s+2}$ in $te^{-t}\cdot \frac{t}{e^t - 1}$, or, the coefficient of $t^{s+1}$ in
\begin{align*}
	e^{-t}\cdot \frac{t}{e^t - 1} &= (e^{-t} - 1) \frac{t}{e^t - 1} + \frac{t}{e^t - 1}\\
	&= -t e^{-t} + \frac{t}{e^t - 1},
\end{align*} 
which is $- \frac{(-1)^s}{s!} + b_{s+2}$.\qed\\

\noindent\textbf{Proof of Lemma \ref{lem:lem9}.} We put
\[
	\delta^{\underline{z}}_{a_1\dotsc a_s} = (\gamma_a^{(w)} - 1)\omega^{\underline{z} w}_{a_1\dotsc a_s bb} - \sum_{0\leq k\leq s} \frac{(-1)^{k+1}}{(k+1)!} \delta_{a a_s\dotsc a_{s-k+1}} \omega^{\underline{z} w}_{a_1\dotsc a_{s-k}a}.
\]
It is checked as in the original proof that $\delta^{\underline{z}}_{a_1\dotsc a_s}$ is a holomorphic 1-form on $C$. In particular, it does not depend on $b$ so we take $b= a$. It remains to show that the integral is zero for any $A$-loop. We take $w\in{\tilde D}$.

Using Lemma \ref{lem:integralout}, we have
\begin{align*}
	\int_{\mathcal{A}_c^{\tilde D}}\delta^{\underline{z}}_{a_1\dotsc a_s} &= \int_{\mathcal{A}_c^{\tilde D}}(\gamma_a^{(w)} - 1)\omega^{\underline{z} w}_{a_1\dotsc a_s aa} - \sum_{0\leq k\leq s} \frac{(-1)^{k+1}}{(k+1)!} \delta_{a a_s\dotsc a_{s-k+1}} \int_{\mathcal{A}_a^{\tilde D}} \omega^{\underline{z} w}_{a_1\dotsc a_{s-k}a}\\
	&= \delta_{a_1\dotsc a_sa c} \left(b_{s+2} - \frac{(-1)^s}{s!}\right) - \delta_{a_1\dotsc a_sa c} b_{s+2}\\
	&\qquad - \sum_{0\leq k\leq s} \frac{(-1)^{k+1}}{(k+1)!} \delta_{a a_s\dotsc a_{s-k+1}} \delta_{a_1\dotsc a_{s-k} a c} b_{s-k+1}\\
 	&=  \delta_{a_1\dotsc a_sa c} \left( - \frac{(-1)^s}{s!} - \frac{(-1)^{s+1}}{s!}\right)\\
	&= 0.
\end{align*}
Indeed, since $\sum_{p\geq 0} \frac{(-1)^{p+1}}{(p+1)!}  t^p = \frac{e^{-t} - 1}{t}$, the last sum is the coefficient of $t^s$ in 
\begin{align*}
	 \frac{e^{-t} - 1}{t}\frac{t}{e^t - 1} = -e^{-t},
\end{align*}
which is $\frac{(-1)^{s+1}}{s!}$. This completes the proof.\qed

\subsection{Proof 2}

Next, we provide a proof based on generating series. Let $[g] = \{1,\dotsc, g\}$ and $\mathbb L(x_1,\dotsc , y_g)$ be the topologically free Lie algebra generated by $x_a,y_a$, $a\in [g]$. Let $t=\sum_a [x_a,y_a]$. 

\begin{definition}
Define a $\mathbb L(x_1,\dotsc ,y_g)$-valued meromorphic section of the sheaf $\Omega^1_{\tilde C}\boxtimes\mathcal O_{\tilde C}$ on $\tilde C\times\tilde C$ by   
\[
	K^{\underline zw} =\sum_{s\geq0} \sum_{a_1,\ldots,a_s,a\in [g]}\omega^{\underline zw}_{a_1\dotsc a_s a}[x_{a_1},\ldots,[x_{a_s},y_b]]. 
\]
\end{definition}

\begin{lemma}\label{lem:prev}
For any $a\in[g]$, one has
\begin{enumerate}[(a)]
	\item $\gamma_a^{(z)}K^{\underline zw}=e^{\ad x_a}(K^{\underline zw})$; 
	\item $\int_{z\in \mathcal A_a}K^{\underline zw}={\ad x_a\over e^{\ad x_a}-1}(y_a) $ whenever $w\in \tilde D$; and
	\item $\mathrm{res}_{z=w}K^{\underline zw}=- \frac{t}{2\pi \sqrt{-1}}$.
\end{enumerate}
\end{lemma}
\noindent Proof.
(a) One has 
\begin{align*}
\gamma_a^{(z)}K^{\underline zw} &= \sum_{s\geq0}\sum_{a_1,\dotsc ,a_s,b\in [g]}\gamma_a^{(z)}\omega^{\underline zw}_{a_1\dotsc a_sb}[x_{a_1},\ldots,[x_{a_s},y_b]]\\
	&= \sum_{s\geq 0} \sum_{a_1,\dotsc ,a_s,b\in [g]}\sum_{k=0}^s \frac{1}{k!} \delta_{aa_1\dotsc a_k} \omega^{\underline zw}_{a_{k+1}\dotsc a_sb}[x_{a_1},\ldots,[x_{a_s},y_b]]\\
	&= \sum_{k,l\geq0}\sum_{a_1,\dotsc ,a_l,b\in [g]} \frac{1}{k!} (\ad x_a)^k \omega^{\underline zw}_{a_{1}\dotsc a_lb}[x_{a_1},\ldots,[x_{a_l},y_b]]\\
	&= e^{\ad x_a}(K^{\underline zw})
\end{align*}
where the second equality follows from the definition of $\omega^{\underline zw}_{a_{1}\dotsc a_sb}$. This proves (a). 

\noindent (b) One has 
\begin{align*}
	\int_{z\in \mathcal A_a}K^{\underline zw} &=\sum_{s\geq0}\sum_{a_1,\dotsc ,a_s,a\in [g]}(\int_{z\in \mathcal A_a}\omega^{\underline zw}_{a_1\dotsc a_sa})[x_{a_1},\ldots,[x_{a_s},y_a]]\\
	&= \sum_{s\geq0}\sum_{a_1,\dotsc ,a_s,a\in [g]}b_s\delta_{aa_1\dotsc a_s}[x_{a_1},\ldots,[x_{a_s},y_a]]\\
	&= \sum_{s\geq0}b_s[x_{a},\ldots,[x_{a},y_a]]\\
	&= \frac{\ad x_a}{e^{\ad x_a} - 1}(y_a). 
\end{align*}
where the second equality follows from Lemma \ref{lem:integral}. This proves (b). 

\noindent (c) One has 
\begin{align*}
	\mathrm{res}_{z=w}K^{\underline zw} &=\sum_{s\geq0}\sum_{a_1,\dotsc ,a_s,a\in [g]}\mathrm{res}_{z=w}\omega^{\underline zw}_{a_1\dotsc a_sa}[x_{a_1},\ldots,[x_{a_s},y_a]]\\
	&= \sum_{s\geq0}\sum_{a_1,\dotsc ,a_s,a\in [g]} -\frac{1}{2\pi \sqrt{-1}}\delta_{s1}\delta_{a_1a}[x_{a_1},\ldots,[x_{a_s},y_a]]\\
	&= -\frac{1}{2\pi \sqrt{-1}}\sum_a[x_a,y_a],  
\end{align*}
where the second equality follows from the definition of $\omega^{\underline zw}_{a_{1}\dotsc a_sb}$. This proves (c). \qed\\

Let $\mathbb L(x_1,\dotsc ,y_g)[1]$ be the part of $\mathbb L(x_1,\dotsc ,y_g)$ of $y$-degree 1, where the $y$-degrees of $x_a,y_a$ are $0,1$. For $w\in\tilde C$, let $K_w = K_w^{\underline z}$ be the 1-form on $\tilde C$ obtained from $K^{\underline zw}$ by pull-back by $\tilde C\to \tilde C^2\colon z\mapsto (z,w)$.

\begin{lemma}\label{lem:uniq}
Let $w\in\tilde D$. Then, for any $u\in \mathbb L(x_1,\dotsc ,y_g)[1]$ of total degree $\geq2$, there is a 
unique $\mathbb L(x_1,\dotsc ,y_g)[1]$-valued differential $L_u^{\underline z}$ on $\tilde C$, 
meromorphic with only pole in $\tilde D$ simple and located at $w$, such that 
\begin{align}\label{eq:conds}
	\gamma_a^{(z)}L_u^{\underline z}=e^{\ad x_a}(L_u^{\underline z}),\quad \mathrm{res}_{z=w}L_u^{\underline z}=-\frac{u}{2\pi \sqrt{-1}}. 
\end{align}
\end{lemma}

\noindent Proof. Let $\mathbb L(x_1,\dotsc ,y_g)[1]_{\geq2}$ be the part of $\mathbb L(x_1,\dotsc ,y_g)[1]$ of total degree $\geq2$. 
The isomorphism
\begin{align*}
	T(x_1,\dotsc ,x_g)^{\oplus g} &\to\mathbb L(x_1,\dotsc ,y_g)[1]\\
	(t_1,\dotsc ,t_g) &\mapsto \sum_a\ad (t_a)(y_a),
\end{align*}
restricts to another isomorphism $T(x_1,\dotsc ,x_g)_{\geq 1}^{\oplus g}\to \mathbb L(x_1,\dotsc ,y_g)[1]_{\geq2}$. 
One proves that the map 
\begin{align}\label{eq:lie2}
	\mathbb L(x_1,\dotsc ,y_g)[1]^{\oplus g} &\to \mathbb L(x_1,\dotsc ,y_g)[1]_{\geq2}\\
	(v_1,\dotsc ,v_g) &\mapsto \sum_a(e^{\ad x_a}-1)(v_a)\nonumber
\end{align}
is a bijection. In fact, composing \eqref{eq:lie2} with these bijections gives rise to the map
\begin{align*}
	T(x_1,\dotsc ,x_g)^{\oplus g^2}\to T(x_1,\dotsc ,x_g)_{\geq 1}^{\oplus g}\\
	(t_{ab})_{1\leq a,b\leq g} \mapsto \big( \sum_a(e^{x_a}-1)t_{ab}\big )_{1\leq b\leq g}.
\end{align*}
This is the $g$-th Cartesian power of the map  
\begin{align*}
	T(x_1,\dotsc ,x_g)^{\oplus g}\to T(x_1,\dotsc ,x_g)_{\geq 1}\\
	(t_a)_a \mapsto \sum_a(e^{x_a}-1)t_a,
\end{align*} 
which is bijective.

Let $u \in \mathbb{L}(x_1,\dotsc ,y_g)[1]$ be of total degree $\geq 2$, and let us prove the existence of $L_u^{\underline{z}}$ satisfying \eqref{eq:conds}. Denote by $(u_1,\dotsc ,u_g)$ the preimage of $u$ by the map \eqref{eq:lie2}. There is a unique endomorphism $\theta_u$ of $\mathbb L(x_1,\dotsc ,y_g)$ preserving the $y$-degree and such that $\theta_u(x_a) =  x_a$ and $\theta_u(y_a) = u_a$, and it follows that $\theta_u(t) = u$. Therefore, by Lemma \ref{lem:prev}, $\theta_u(K^{\underline z}_w)$ is a solution of \eqref{eq:conds}. 

Let us prove the uniqueness of a solution of \eqref{eq:conds}. If $L_u^{\underline z}$ is such a solution, the integrals $I_a:=\int_{z\in \mathcal A_a}L_u^{\underline z}$ belong to $\mathbb L(x_1,\dotsc ,y_g)[1]$, and the residue formula in Section 4.1 of \cite{enriquez} implies, together with the automorphy properties \eqref{eq:conds}, the relation $\sum_a(e^{\ad x_a}-1)(I_a)=u$. By the above-mentioned bijectivity, this implies $I_a=u_a$ for any $a$. Then $X^{\underline z}:=L_u^{\underline z}-\theta_u(K^{\underline z}_w)$ satisfies 
\begin{align*}
	\gamma_a^{(z)}X^{\underline z} =e^{\ad x_a}(X^{\underline z}),\quad 
\int_{z\in \mathcal A_a}X^{\underline z}= 0,\quad \mathrm{res}_{z=w}X^{\underline z}=0.  
\end{align*}
If $X^{\underline z}$ is nonzero, then its nonzero component $x^{\underline z}$ of lowest degree 
(assigning degree 1 to $x_a,y_a$) satisfies 
\[
	\gamma_a^{(z)}x^{\underline z}=x^{\underline z},\quad 
\int_{z\in \mathcal A_a}x^{\underline z}= 0,\quad \mathrm{res}_{z=w}x^{\underline z}=0.  
\]
The first and last equalities imply that $x^{\underline z}$ is a holomorphic differential, and the second equality then implies 
its vanishing; this is a contradiction, hence $X^{\underline z}=0$, proving the claimed uniqueness. \qed

\begin{definition}
For $a\in[g]$, $\theta_a$ is the automorphism of $\mathbb L(x_1,\dotsc ,y_g)$ defined by 
\[	
	\forall b\in[g]:\quad  \theta_a(x_b)=x_b,\quad \theta_a(y_b)=y_b+\delta_{ab}{{e^{\ad x_a}-1}\over\ad x_a}(t). 
\]
\end{definition}
One has
\[
	\theta_a(t)=\sum_b [x_b,y_b+\delta_{ab}{{e^{\ad x_a}-1}\over{\ad x_a}}(t)]
=t+[x_a,{{e^{\ad x_a}-1}\over{\ad x_a}}(t)]=e^{\ad x_a}(t).
\]
One can check that $\theta_a$ is the unique automorphism of $\mathbb L(x_1,\dotsc ,y_g)$ such that $x_b\mapsto x_b$ for any $b$, and $t\mapsto e^{\ad x_a}(t)$. 

\begin{lemma}\label{lem:last}
For any $a\in[g]$, one has 
\[
	\gamma_a^{(w)}K^{\underline zw}=\theta_a^{-1}(K^{\underline zw}). 
\]
\end{lemma}
\noindent Proof.
We prove the pull-back of this equality by the maps $\tilde C\to\tilde C^2\colon z\mapsto (z,w)$ 
for $w\in\gamma_a(\tilde D)$, which implies the equality by analytic prolongation. The desired equality is then equivalent to 
\[
	\theta_a(K_{\gamma_a^{-1}(w)})=K_w, 
\]
which we now prove. Both sides of this equality are 1-forms on $\tilde C$, 
they are $\mathbb L(x_1,\dotsc ,y_g)[1]$-valued as $\theta_a$ preserves the $y$-degree, 
and they both satisfy the first equality of Lemma \ref{lem:uniq} as $\theta_a$ leaves $x_1,\dotsc ,x_g$ invariant, and $\gamma_a^{(z)}$ commutes with $\gamma_b^{(w)}$. Also, 
\begin{align*}
	\mathrm{res}_{\gamma_a^{-1}(w)}(\theta_a(K_{\gamma_a^{-1}(w)})) &= \theta_a(\mathrm{res}_{\gamma_a^{-1}(w)}K_{\gamma_a^{-1}(w)})\\
	&=-\frac{1}{2\pi \sqrt{-1}}\theta_a(t)\\
	&=-\frac{1}{2\pi \sqrt{-1}}e^{\ad x_a}(t)\\
	&=\mathrm{res}_{\gamma_a^{-1}(w)}(K_w),
\end{align*}
where the last equality follows from 
\[
	\mathrm{res}_{\gamma_a^{-1}(w)}(K_w)=\mathrm{res}_{w}(K_w\circ \gamma_a^{-1}) =\mathrm{res}_{w}e^{\ad x_a}(K_w)=-(2\pi \sqrt{-1})^{-1}e^{\ad x_a}(t). 
\]
The statement then follows from Lemma \ref{lem:uniq}. \qed\\

\noindent\textbf{Proof of Lemma \ref{lem:lem9}.}
One has $\theta_a^{-1}(y_b)= y_b + \delta_{ab}{{e^{-\ad x_a}-1}\over\ad x_a}(t)$, therefore 
\begin{align*}
	&\sum_{s\geq0}\sum_{a_1,\dotsc ,a_s,b\in [g]}\gamma_a^{(w)}\omega^{\underline zw}_{a_1\dotsc a_sa}[x_{a_1},\ldots,[x_{a_s},y_b]]\\
	&=\gamma_a^{(w)}K^{\underline zw} =\theta_a^{-1}(K^{\underline zw})\\
	&= \sum_{s\geq0}\sum_{a_1,\dotsc ,a_s,b\in [g]}\omega^{\underline zw}_{a_1\dotsc a_sa}\theta_a^{-1}([x_{a_1},\ldots,[x_{a_s},y_b]])\\
	&= \sum_{s\geq0}\sum_{a_1,\dotsc ,a_s,b\in [g]}\omega^{\underline zw}_{a_1\dotsc a_s,b}[x_{a_1},\ldots,[x_{a_s},y_b+\delta_{ab}{{e^{-\ad x_a}-1}\over\ad x_a}(t)]]\\
	&= \sum_{s\geq0}\sum_{a_1,\dotsc ,a_s,b\in [g]}\omega^{\underline zw}_{a_1\dotsc a_s,b}[x_{a_1},\ldots,[x_{a_s},y_b]]\\
	&\qquad +\sum_{s\geq0}\sum_{a_1,\dotsc ,a_s,c\in [g]}\omega^{\underline zw}_{a_1\dotsc a_s}[x_{a_1},\ldots,[x_{a_s},{{e^{-\ad x_a}-1}\over\ad x_a}[x_c,y_c]]]
\end{align*}
where the second equality follows from Lemma \ref{lem:last}; this implies the claim (as well as the $F_g$-invariance part of Lemma 9 (a) in \cite{enriquez}). \qed\\

\small
\bibliographystyle{alphaurl}
\bibliography{dhgKZB_arxiv.bib}

\end{document}